\documentclass[review]{elsarticle}

\usepackage{inputenc}
\usepackage[T1]{fontenc}
\usepackage{textcomp}
\usepackage[english]{babel}
\usepackage[bitstream-charter]{mathdesign}
\usepackage{amsmath, amsopn,amscd,amsthm}
\usepackage{dsfont}
\usepackage{algorithm2e}
\usepackage{graphicx}
\usepackage{enumerate}
\usepackage{stmaryrd}
\usepackage{array}
\usepackage{float}
\usepackage{pgf, tikz}

\usepackage{hyperref}
\hypersetup{colorlinks=true,linkcolor=blue,citecolor=blue,linktoc=page}

\renewcommand{\ni}{\noindent}
\newcommand{\eq}{\begin{equation}}
\newcommand{\qe}{\end{equation}}
\renewcommand{\d}{\mathrm{d}}

\newcommand{\measure}{\mathds{M}}
\newcommand{\normTV}[1]{\left\lVert #1\right\lVert_{TV}}
\newcommand{\abs}[1]{\left\vert#1\right\vert}

\theoremstyle{plain}% default
\newtheorem{Assumption}{Assumption}
\newtheorem{thm}{Theorem}%[section]
\newtheorem{lem}{Lemma}

\newtheorem{defn}{Definition}
\theoremstyle{definition}

\theoremstyle{remark}
\newtheorem*{rem}{Remark}

\newcommand{\mbf}{\mathbf}

\newcommand{\ind}{\mathds{1}}
\newcommand{\R}{\mathds{R}}                     %real numbers
\def\P{\mathds{P}}
\def\E{\mathds{E}} 
\def\Var{\mathop{\rm Var}}

\journal{J. of Mathematical Analysis and Applications}

\begin{document}

\begin{frontmatter}
\sloppy

\title{Non-uniform spline recovery from small degree polynomial approximation}

% Authors
\author{Yohann De Castro}
%\address{D\'epartement de Math\'ematiques (CNRS UMR 8628), B\^atiment 425, Facult\'e des Sciences d'Orsay, Universit\'e Paris-Sud 11, F-91405 Orsay Cedex, France.}
\ead{yohann.decastro@math.u-psud.fr}
\ead[url]{www.math.u-psud.fr/$\sim$decastro} 
\author{Guillaume Mijoule}
\address{D\'epartement de Math\'ematiques (CNRS UMR 8628), B\^atiment 425, Facult\'e des Sciences d'Orsay, Universit\'e Paris-Sud 11, F-91405 Orsay Cedex, France.}
\ead{guillaume.mijoule@math.u-psud.fr}

%abstract%%%%%%%%%%%%%%%%%%%%%%%%%%%%%%%%%%%%%%%%%%%%%%%%%%%%%%%%%%%%%%%%%%%%%%
\begin{abstract}
We investigate the sparse spikes deconvolution problem onto spaces of algebraic polynomials. Our framework encompasses the measure reconstruction problem from a combination of noiseless and noisy moment measurements. We study a TV-norm regularization procedure to localize the support and estimate the weights of a target discrete measure in this frame. Furthermore, we derive quantitative bounds on the support recovery and the amplitudes errors under a Chebyshev-type minimal separation condition on its support. Incidentally, we study the localization of the knots of non-uniform splines when a Gaussian perturbation of their inner-products with a known polynomial basis is observed (i.e. a small degree polynomial approximation is known) and the boundary conditions are known. We prove that the knots can be recovered in a grid-free manner using semidefinite programming.
\end{abstract}
%%%%%%%%%%%%%%%%%%%%%%%%%%%%%%%%%%%%%%%%%%%%%%%%%%%%%%%%%%%%%%%%%%%%%%

\begin{keyword}
LASSO \sep Super-resolution \sep Non-uniform splines \sep Algebraic polynomials \sep $\ell_{1}$-minimization.
\end{keyword}
%classification : 62F10; 65D07; 65D15; 65D05

\end{frontmatter}

%%%%%%%%%%%%%%%%%%%%%%%%%%%%%%%%%%%%%%%%%%%%%%%%%%%%%%%%%%%%%%%%%%%%%%
\section{Introduction}
%%%%%%%%%%%%%%%%%%%%%%%%%%%%%%%%%%%%%%%%%%%%%%%%%%%%%%%%%%%%%%%%%%%%%%

\subsection{Non-uniform spline recovery}

Our framework involves the recovery of non-uniform splines, i.e. a smooth polynomial function that is piecewise-defined on subintervals of different lengths. More precisely, we investigate a grid-free procedure to estimate a non-uniform spline from a polynomial approximation of small degree. Our estimation procedure can be used as a post-processing technique in various fields such as data assimilation \cite{kalnay2003atmospheric}, shape optimization \cite{henrot2006variation} or spectral methods in PDE's \cite{gottlieb1977numerical}. 

For instance, one gets a polynomial approximation of the solution of a PDE when using spectral methods such as the Galerkin method. In this setting, one seeks a weak solution of a PDE using bounded degree polynomials as test functions. Then, the Lax-Milgram theorem grants the existence of a unique weak solution $\mbf f$ for which a polynomial approximation $P$ can be computed. Moreover, C\'ea's lemma shows that the Galerkin approximation $P$ is comparable to the best polynomial approximation $\mbf p(\mbf f)$ of the weak solution $\mbf f$. This situation can be depicted by Assumption \ref{eq:AssumptionSplines}. Hence, if one knows the weak solution $\mbf f$ is a non-uniform spline then our (post-processing) procedure can provide a grid-free estimate $\mbf f$ from the Galerkin approximation $P$. Moreover, Theorem \ref{thm:Spline} shows that the recovered spline has large discontinuities near the large discontinuities of the target spline $\mbf f$. Hence, the location of the large enough discontinuities of the weak solution $\mbf f$ can be quantitatively and in a grid-free manner estimated from the Galerkin approximation using our algorithm.

As an example, Figure \ref{fig:Example} illustrates how our procedure improves a polynomial approximation of a non-uniform spline. Observe that discontinuities of splines make them difficult to approximate by polynomials. Consider an approximation (thin black line) of the spline (thick dashed gray line). It seems rather difficult to localize the discontinuities of the spline from the knowledge of this polynomial approximation and the boundary conditions. Nevertheless, our procedure produces a non-uniform spline (thick black line) whose large discontinuities are close to the knots of the target spline.

\begin{figure}%[h!]
\begin{center}
\includegraphics[width=0.7\textwidth]{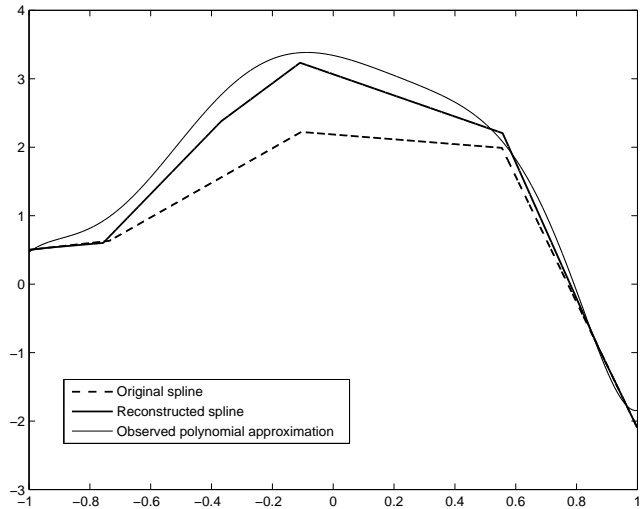}
\end{center}
\caption{Estimated spline (thick black line) of a non-uniform spline (thick dashed gray line) and its knots from a polynomial approximation (thin black line).}
\label{fig:Example}
\end{figure}

The method we propose is as follows. Following an idea of \cite{bendory2013exact}, we aim at reconstructing a spline of degree $d$ by recovering its $d+1$ distributional derivative, using tools of the super-resolution theory  \cite{donoho1992superresolution,candes2012towards,de2012exact}. More precisely, consider an univariate spline $\mbf f$ of degree $d$, defined on $[-1,1]$. The $d+1$ distributional derivative of $\mbf f$, denoted $\mbf f^{(d+1)}$, is a discrete signed measure whose support are the knots of the spline. Using an integration by parts, one can show that the first $m+1$ polynomial moments of $\mbf f^{(d+1)}$ can be expressed as a linear combination of the first $m-d$ moments of $\mbf f$ and its $2(d+1)$ boundary conditions; moreover, the first $d+1$ moments of  $\mbf f^{(d+1)}$ only depend on the boundary conditions (details are in Lemma \ref{lem:Matrix}). As a consequence, observing $m-d$ noisy moments and the (noiseless) boundary conditions of $\mbf f$ is equivalent to observing $d+1$ noiseless and $m-d$ noisy moments of $\mbf f^{(d+1)}$. This observation is the motivation of the theoretical work of this paper.

%%%%%%%%%%%%%%%%%%%%%%%%%%%%%%%%%%%%%%%%%%%%%%%%%%%%%%%%%%%%%%%%%%%%%%
\subsection{Sparse spikes deconvolution onto spaces of algebraic polynomials}
%%%%%%%%%%%%%%%%%%%%%%%%%%%%%%%%%%%%%%%%%%%%%%%%%%%%%%%%%%%%%%%%%%%%%%

In this paper, we extend some recent results in spike deconvolution to the frame of algebraic polynomials. Beyond the theoretical interest, we focus on this model in order to bring tools and quantitative guarantees from the super-resolution theory to the companion problem of the recovery of knots of non-uniform splines \cite{bendory2013exact}. At first glance, this setting can be depicted as a deconvolution problem where one wants to recover the location of the support of a discrete measure from the observation of its convolution with an algebraic polynomial of given degree $ m$. More precisely, we aim at recovering a discrete measure from the knowledge of the true $( d+1)$ first moments and a noisy version of the $( m- d)$ next ones.

\subsection{Previous works}

The super-resolution problem has been intensively investigated in the last years. In \cite{bredies2013inverse,de2012exact} the authors give an exact recovery condition for the noiseless problem in a general setting. In the Fourier frame, this analysis was greatly refined in \cite{candes2012towards} which shows that the exact recovery condition is satisfied for all measure satisfying a ``minimum separation condition''. The recovery from noisy samplings was investigated in \cite{candes2012super} which characterizes the reconstruction error as the resolution increases. The first results on quantitative localization was brought by the authors of \cite{ azais2014spike} who give bounds on the support detection error in a general frame. This analysis was derived in terms of the amplitude of the target measure in \cite{fernandez2013support}. In the Fourier frame, the optimal rates in prediction error have been investigated in \cite{tang2013near}.
Lastly, the behavior and the stability of $TV$-norm regularization in the space of measures has been investigated in \cite{duval2013exact} when observing small noise errors.

The spline recovery problem in the noiseless case has been studied in \cite{bendory2013exact} where the authors assume that one knows the orthogonal projection $\mbf p(\mbf f)$ of the non-uniform spline $\mbf f$. Our frame extends their point of view to the noisy case where one observes a polynomial approximation $P$% close to the best polynomial approximation $\mbf p(\mbf f)$
. To the best of our knowledge, there is no result on a quantitative localization of the knots of non-uniform splines from noisy measurements. 

%%%%%%%%%%%%%%%%%%%%%%%%%%%%%%%%%%%%%%%%%%%%%%%%%%%%%%%%%%%%%%%%%%%%%%
\section{General model and notation}
%%%%%%%%%%%%%%%%%%%%%%%%%%%%%%%%%%%%%%%%%%%%%%%%%%%%%%%%%%%%%%%%%%%%%%
Let $[-1,1]$ be equipped with the distance:
\[
\forall\,u\,,v\in[-1,1]\,,\quad\mathit d(u,v)=|\arccos u - \arccos v|\,.
\]
Let $\mathbf x$ be a signed measure on ${[-1,1]}$ with finite support of unknown size $s$. In particular, $\mbf x$ admits a polar decomposition:
\begin{equation}\label{def:TargetMeasure}
\mbf x=\sum_{k=1}^s a_k\,\delta_{ t_k}\,,
\end{equation}
 where $a_k\in\mathds R\setminus\{0\}$, $t_k\in {[-1,1]}$, and
$\delta_t$ denotes the Dirac measure at point $t$. Let $ m$ be a positive integer and $\mathscr F=\{\varphi_0,\varphi_1,\dotsc,\varphi_{ m}\}$ be such that $\varphi_{0}=1$ and for $k=1,\ldots,{ m}$, \[\varphi_{k}=\sqrt 2\,T_{k}\,,\] where $T_{k}(t)=\cos(k\arccos(t))$ is the $k$-th Chebyshev polynomial of the first kind. Observe that the family $\mathscr F$ is an orthonormal family with respect to the probability measure $\measure(\d t)=(1/\pi)\,({1-t^{2}})^{-1/2}\,\mathcal L(\d t)$ on $[-1,1]$, where $\mathcal L$ denotes the Lebesgue measure. Define the $k$-th {generalized moment} of a signed measure $\mu$ on ${[-1,1]}$ as:
\[
c_{k}(\mu)=\displaystyle\int\nolimits_{[-1,1]} \varphi_{k}\,\d\mu\,, 
\]
for $k=0,1,\dotsc,{ m}$. Assume that we observe $c_k(\mbf x)$ for $0\leq k\leq d$ and a noisy version of $c_k(\mbf x)$ for $ d+1\leq k\leq { m}$, where possibly $ d=-1$. Define $y_k=c_k(\mbf x)+\varepsilon_k$ such as $\varepsilon_{k}=0$ for $0\leq k\leq d$ and $\varepsilon_{k}$ are i.i.d. $\mathcal N(0,\sigma^{2})$ for $ d+1\leq k\leq { m}$. This can be written as:
\eq\label{eq:defObservation}
\mbf y=\mbf c(\mbf x)+\mbf e\,,
\qe
where $\mbf c(\mbf x)=(c_{k}(\mbf x))_{k=0}^{ m}$ and $\mbf e=(0, \mbf n)$ with $\mbf n\sim\mathcal N(0,\sigma^{2}\,\mathrm{Id}_{ m- d})$. Note we know the first true moments up to the order $ d$ and a noisy version of them up to the order $ m$. Moreover, the degree $ d$ is allowed to be $-1$.

\subsection{An L1-minimization procedure}
Our analysis follows recent proposals on $\ell_{1}$-minimization \cite{ bredies2013inverse,azais2014spike, tang2013near, duval2013exact}. Denote by $\mathcal M$ the set of all finite signed measures on ${[-1,1]}$ endowed with the  {total variation} norm $\normTV{\,.\,}$, which is isometrically isomorphic to the dual $\mathscr C([-1,1])^\star$ of continuous function endowed with the supremum norm. We recall that for all $\mu\in \mathcal M$, \[\normTV{\mu}=\sup_{\mathscr P}\sum_{E\in\mathscr P}\abs{\mu(E)}\,,\] where the supremum is taken over all partitions $\mathscr P$ of ${[-1,1]}$ into a finite number of disjoint measurable subsets. Consider a modified version of the convex program \textrm{BLASSO} \cite{ azais2014spike} given by:
\eq
\label{def:Blasso}
 \hat{\mbf x}\in\arg\min_{ \mu\in\mbf C_{ d}(\mbf x)}\frac12\lVert \mbf c(\mu)-\mbf y\lVert^2_2+\lambda\lVert\mu\lVert_{TV}\,,
\qe
where $\mbf C_{ d}(\mbf x):=\{\mu\in \mathcal M\,;\quad \forall\,k=0,\ldots, d\,,\ c_{k}(\mu)=c_{k}(\mbf x)\}$ and $\lambda>0$ is a tuning parameter. Questions immediately arise:

\begin{itemize}
	%\item Can we compute a solution to \eqref{def:Blasso}?
 	\item How close is the recovered spike measure from the target $\mbf x$?
  	\item How accurate is the localization of \eqref{def:Blasso} in terms of the noise and the amplitude of the recovered/original spike?
\end{itemize}

\ni
To the best of our knowledge, this paper is the first to quantitatively address these questions in the frame of algebraic polynomials.

\subsection{Contribution}
\begin{defn}[Minimum separation]
Let $\mbf T \subset [-1,1]$. We define $\Delta(\mbf T)$, the minimum separation of $\mbf T$, by
\[
\Delta(\mbf T)=\min_{(t,t')\in \mbf T^2; t\neq t'} \min\ \{d(t,t'), \pi-d(t,t')\}\,,
\]
that is the minimum modulus between two points of $\arccos(\mbf T)+\pi\mathds Z$.
\end{defn}
\ni
Let $\epsilon(\mbf T)$ denote the distance from $\mbf T\backslash  \{ -1,1\}$ to the edges of $[-1,1]$:
\[
\epsilon(\mbf T) = \min \,\{ \min(d(t,1),d(t,-1)); \; t \in \mbf T\backslash  \{ -1,1\} \} \,.
\]
\begin{thm}\label{thm:Main}
Assume $m \geq 128$. Let $\eta>0$ and set:
\[
\lambda_{0}:=2\sigma[2(1+\eta)( m- d)\log(5( m+ d+1))]^{1/2}\,,
\]
then with probability greater than $1-\big[\frac1{5( m+ d)}\big]^{\eta}$ the following holds. If  $\lambda\geq\lambda_{0}$ and
\eq
\label{hyp:separation}
\min\{\Delta (\mbf T),\,2\epsilon(\mbf T)\} \geq \frac{5\pi}{ m},
\qe
then there exists a solution $\hat{\mbf x}$ to \eqref{def:Blasso} with finite support $\displaystyle\hat{\mbf x}=\sum_{k=1}^{\hat s}\hat{a}_{k}\delta_{\hat{t}_{k}}$ satisfying:
\begin{enumerate}[(i)]
\item Global control:
\[
\displaystyle\sum_{k=1}^{\hat s}|\hat{ a}_{k}|\min\Big\{ m^{2}\min_{t\in\mbf T}d(t, {\hat{t}_{k}})^{2};c_{0}^{2}\Big\}\leq{c_{1}}\lambda\,,
\]
\item Local control: 
\[
\forall i=1,\ldots,s,\quad\displaystyle\Big| a_{i}-\sum_{\hat t_k \in\mathrm{Supp}(\hat{\mbf x})\, | \, d( t_{i},\hat t_k) \leq \frac{c_{0}}{ m}} \hat{a}_k \Big|\leq c_{2}\lambda\,,
\]
\item Large spike localization:
\[
\forall i=1,\ldots,s,\ \mathrm{s.t.}\ |a_{i}|>c_{2}\lambda\,,\ \exists\,\hat{ t}\in\mathrm{Supp}(\hat{\mbf x})\ \mathrm{s.t.}\ \displaystyle d({t}_{i},\hat{t})\leq \Big[\frac{c_{1}\lambda}{|a_{i}|-c_{2}\lambda}\Big]^{1/2}\frac1{ m}\,,
\]
\end{enumerate}
where $c_{0}=1.0361$, $c_{1}=235.85$, and $c_{2}=220.72$.
\end{thm}
\ni
In the proof of the theorem, we will need the two following lemmas, which capitalize on the recent papers \cite{candes2012towards} and build an explicit dual certificate in the frame of algebraic polynomials. More precisely, we explicitly bound from above the dual certificates by a quadratic function near the support points, as done in \cite{candes2012towards}.

\begin{lem}\label{lem:DualCertificate}
Assume \eqref{hyp:separation} holds. Then  for all $t_j \in \mbf T$, there exists a polynomial $q_{ t_j}$ of degree $ m$ such that:
\begin{enumerate}
\item $q_{t_j}( t_j) = 1$,  
\item $\forall t_l \in \mbf T \backslash\{t_j\}\,,\quad q_{t_j}(t_l) = 0$,\label{thm:enu:1}
\item if $d(t, t_j) \leq c_{0} / m$ then:\label{thm:enu:2}
 \[1-C_{2}\, m^2 d(t, t_j)^2\leq q_{ t_j}(t) \leq 1-C_{1}\, m^2 d(t, t_j)^2\,,\]
\item if $d(t,t_l) \leq c_{0} / m$ and $t_l \in \mbf T \backslash\{ t_j\}$ then:\label{thm:enu:3}
\[C_{1}\, m^2 d(t, t_j)^2\leq q_{ t_j}(t) \leq C_{2}\, m^2 d(t, t_j)^2\,,\]
\item if $d(t,t_l) > c_{0} / m$ for all $t_l \in \mbf T$ then: \label{thm:enu:4}
\[c_{0}^{2}C_{1}  \leq q_{ t_j}(t) \leq 1-c_{0}^{2}C_{1}\,,\]
\end{enumerate}
where $c_{0}=2\pi\cdot 0.1649$, $C_{1}=0.00424$, and $C_{2}=0.25$.
\end{lem}
\begin{proof}
By symmetrizing the support, we can use existing results for real trigonometric polynomials. Let $X = \frac{1}{2\pi}\left( \arccos(\mbf T) \bigcup [-\arccos(\mbf T)] \right) +\frac{1}{2}$. Note that $X \subset [0,1]$. It is easy to check that \eqref{hyp:separation} implies:
\eq
\min_{(x,x') \in X\, ; \, x \neq x'} |x-x'| \geq 2.5/ m
\qe
Thus, according to Proposition 2.1 and Lemma 2.5 of \cite{candes2012towards}, for all $x_j \in X$, there exists a real trigonometric polynomial of degree $ m$, $\tilde q_{x_j}\, : \, x \mapsto \sum_{k=- m}^{ m} c_k e^{2i\pi k x}$, such that:
\begin{itemize}
\item $\tilde q_{x_j} (x_j) = \tilde q_{x_j} (-x_j) = 1$,
\item $|\tilde q_{x_j} (x) |<1, x \in [0,1] \backslash X$,
\item $\tilde q_{x_j} (x_l) = -1, \quad x_l \in X \backslash \{ x_j, -x_j\}$,
\item $\forall(x,x_l)\in [0,1]\times X\,,\ |x-x_l| \leq 0.1649 / m$, 
\[|\tilde q_{x_j} (x)| \leq 1- 0.3353\,  m^2 (x-x_l)^2\,,\]
\item $\forall x\in [0,1]\,, \ \forall x_l \in X, |x-x_l| > 0.1649 / m$, 
\[|\tilde q_{x_j} (x)| \leq 1- 0.3353 \cdot 0.1649^2\,.\]
\end{itemize}
We stress that the polynomial $\tilde q_{x_j}$ as constructed in Lemma 2.2 of \cite{candes2012towards} is even. We detail the argument here. Let $K$ stand for the square of the Fej\'er kernel, defined by
\[
K(t) = \left[ \frac{\sin\left( ( m/2 +1)\pi t \right)}{( m/2 +1)\sin(\pi t)} \right]^4.
\]
Then, in the proof of Lemma 2.5 in \cite{candes2012towards}, it is shown that there exists  a unique polynomial of the form
\eq
\label{eq:defq}
q (t) = \sum_{x_i \in X} \alpha_i K(t-x_i) + \beta_i K'(t-x_i)
\qe
satisfying
\begin{align}
 q(x_j) = q(-x_j)& = 1, \notag\\
q(x_i) &= -1,\quad\forall x_i \in X\backslash \{ x_j,-x_j\}, \label{eq:condq} \\
q'(x_i) &= 0,\quad \forall x_i \in X, \notag
\end{align}
where $\alpha_i$ and $\beta_i$ are complex numbers.
%Indeed, in this case, the coefficients $\{ \alpha_i, \beta_i\}$ are the solution of an invertible system.
Using the symmetry of $K$, the anti-symmetry of $K'$ and the symmetry of $X$, we see that the polynomial $\tilde q := t\mapsto q(-t)$ is also of the form \eqref{eq:defq}. Using again the symmetry of $X$, we have that $\tilde q$ satisfies \eqref{eq:condq}. By unicity, $\tilde q = q$.

Thus, the trigonometric polynomial function $p_{x_j} \, : \, x \in [-\pi,\pi] \mapsto \tilde q_{x_j}\left( \frac{1}{2\pi}x + \frac{1}{2} \right)$ is real and even, so we have the expansion:
\[
p_{x_j}(x) = \sum_{k=0}^{ m} a_k \cos (k x)\,.
\]
Moreover, since $\displaystyle \sup_{x \in [0,2\pi]}| p_{x_j} (x) | = 1$, Bernstein's inequality \cite{borwein1994chebyshev} implies:
\eq
\label{eq:bernsteinp}
 \sup_{x \in [0,1]}| p_{x_j}''(x) | \leq   m^2\,.
\qe
Let $ t_j \in \mbf T$ and $x_j = \arccos ( t_j)$. We define:
\[
q_{ t_j}(t) = \frac{1}{2} p_{x_j}(\arccos t) +\frac{1}{2}= \frac{1}{2}\sum_{k=0}^m a_k T_k(t) +\frac{1}{2}\,,
\]
where $T_k$ is the k-th Chebyshev polynomial of the first kind. Lemma \ref{lem:DualCertificate} is a direct consequence of the properties verfied by $\tilde q_{x_j}$ and \eqref{eq:bernsteinp}.
\end{proof}
\ni
\begin{lem}\label{lem:QIC}
Assume \eqref{hyp:separation} holds. Then  for all $(v_{1},\ldots,v_{S})$ such that $|v_{j}| = 1$, there exists a polynomial $q$ of degree $ m$ such that:
\begin{enumerate}
\item $\forall j \in [1,S], q( t_j) = v_{j}$,  
\item if $d(t, t_j) \leq c_{0} / m$ then:\label{thm2:enu:2}
 \[ 1-| q(t) | \geq 2C_{1}\, m^2 d(t, t_j)^2\,,\]
\item if $d(t,t_l) > 2\pi \cdot 0.1649 / m$ for all $t_l \in \mbf T$ then: \label{thm2:enu:3}
\[ 1-|q(t)| \geq 2c_{0}^{2}C_{1}\,,\]
\end{enumerate}
where $c_{0}=2\pi\cdot 0.1649$ and $C_{1}=0.00424$.
\end{lem}
\begin{proof}
Similarly as previous lemma, if $X = \frac{1}{2\pi}\left( \arccos(\mbf T) \bigcup [-\arccos(\mbf T)] \right) +\frac{1}{2}$, then we can construct a trigonometric polynomial $\tilde q\, : \, x \mapsto \sum_{k=- m}^{ m} c_k e^{2i\pi k x}$, such that:
\begin{itemize}
\item $\tilde q(x_j) = \tilde q(-x_j) = v_{j}, \forall j \in [1,S]$,
\item $|\tilde q (x) |<1, \forall x \in [0,1] \backslash X$,
\item $\forall(x,x_l)\in [0,1]\times X\,,\ |x-x_l| \leq 0.1649 / m$, 
\[|\tilde q (x)| \leq 1- 0.3353\,  m^2 (x-x_l)^2\,,\]
\item $\forall x\in [0,1]\,, \ \forall x_l \in X, |x-x_l| > 0.1649 / m$, 
\[|\tilde q (x)| \leq 1- 0.3353 \cdot 0.1649^2\,.\]
\end{itemize}
Then $p \, : \, x \in [-\pi,\pi] \mapsto \tilde q\left( \frac{1}{2\pi}x + \frac{1}{2} \right)$ is even, so we have the expansion $p(x) = \sum_{k=0}^{ m} a_k \cos (k x)$ where $a_{k} \in \mathds R$.
Putting
\[
q \, : \, t \mapsto\sum_{k=0}^{ m} a_k \cos (k \arccos t)  =\sum_{k=0}^m a_k T_k(t)\,,
\]
we can show $q$ verifies the needed properties.
\end{proof}

\begin{proof}[Proof of Theorem \ref{thm:Main}]

We mention that the proof of (ii), which uses Lemma \ref{lem:DualCertificate}, follows the one in \cite{fernandez2013support}.

\ni
Assume that $\lambda\geq\lambda_{0}$ where $\lambda_{0}$ is described by the following lemma (the dependence in $\eta$ has been omitted). 

\begin{lem}
\label{lem:Rice}
Set $\lambda_{R}:=\sigma[8( m- d)\log(5( m+ d+1))]^{1/2}$ and $\lambda>\lambda_{R}$, then:
\eq
\notag
\P\left(\bigg\lVert \sum_{k=0}^{ m}\varepsilon_{k}\varphi_{k}\bigg\lVert_{\infty}>\lambda\right)\leq\exp\left[-\frac{\lambda^{2}-\lambda_{R}^{2}}{8\sigma^{2}( m- d)}\right]\,.
\qe
In particular, for all $\eta>0$, if
\[
\lambda_{0}(\eta):=\sigma[8(1+\eta)( m- d)\log(5( m+ d+1))]^{1/2}\,,
\]
then
\eq
\label{eq:ProbaRice}
\P\left(\bigg\lVert \sum_{k=0}^{ m}\varepsilon_{k}\varphi_{k}\bigg\lVert_{\infty}>\lambda_{0}(\eta)\right)\leq\frac1{[5( m+ d+1)]^{\eta}}\,.
\qe
\end{lem}

\ni
A proof of Lemma \ref{lem:Rice} can be found in Appendix \ref{app:Rice}. Observe that the condition of the following lemma is met.

\begin{lem}
\label{lem:MajorationPrediction}
Let $\hat{\mbf x}$ be a solution to \eqref{def:Blasso}. Then the following holds:
\eq
\label{eq:MajorationPrediction}
\forall P\in\mathrm{Span}(\mathscr F)\,,\quad|\int_{-1}^1P\mathrm{d}(\hat{\mbf x}-\mbf x)|%\leq(\lambda+\lambda_{0})\int_{-1}^1|P|\mathrm d\mathbf \Pi
\leq(\lambda+\lambda_{0})\lVert P\lVert_{\infty}\,,
\qe
where $\displaystyle\lambda_{0}\geq\lVert \sum_{k=0}^{ m}\varepsilon_{k}\varphi_{k}\lVert_{\infty}$.
\end{lem}

 \ni
A proof of Lemma \ref{lem:MajorationPrediction} can be found in Appendix \ref{app:MajorationPrediction}. One can prove that there exists a solution $\hat{\mbf x}$ to \eqref{def:Blasso} with finite support, see Lemma \ref{lem:SolDiscrete}. Set: 
\[
\hat{\mbf x}=\sum_{k=1}^{\hat s}\hat{a}_{k}\delta_{\hat{t}_{k}}\,.
\]
Set $v_{j}=\overline{a_{j}}/|a_{j}|$ for $j=1,\ldots,s$ and consider $q=\sum_{k=0}^{ m}\beta_{k}\varphi_{k}$ the algebraic polynomial described in Lemma \ref{lem:QIC}. Set:
\[
\mathscr D:=\lVert\hat{\mbf x}\lVert_{TV}-\lVert\mbf x\lVert_{TV}-\int_{-1}^{1}q\mathrm d(\hat{\mbf x}-\mbf x)\,.
\]
Note that $\mathscr D\geq 0$. Since $\mbf x$ is feasible, it holds:
\[
\frac12\lVert \mbf c(\hat{\mbf x})-\mbf y\lVert^2_2+\lambda\mathscr D+\lambda\int_{-1}^{1}q\mathrm d(\hat{\mbf x}-\mbf x)\leq\frac12\lVert \mbf e\lVert^2_2\,.
\]
Hence, using the fact that for any $\mu\in \mathcal M$, $\langle \mbf c(\mu), \beta \rangle = \int_{-1}^1  q d\mu$,
\[
\frac12\lVert \mbf c(\hat{\mbf x})-\mbf y+\lambda \beta\lVert^2_2+\lambda\mathscr D\leq\frac12\lVert \mbf e\lVert^2_2+\frac12\lVert\lambda \beta\lVert_{2}^{2}-\lambda\langle\mbf e,\beta \rangle\,.
\]
Eventually,
\[
\mathscr D\leq \frac{\lambda}2\bigg\lVert \beta-\frac{\mbf e}\lambda\bigg\lVert_{2}^{2}\,.
\]
Using Lemma \ref{lem:Rice}, we have with probability greater than $1-\frac1{[5( m+ d+1)]^{\eta}}$:
\[
\bigg\lVert\sum_{k=0}^{ m}(\beta_{k}-\varepsilon_{k}/\lambda)\varphi_{k}\bigg\lVert_{\infty}\leq 2\,,
\]
so that:
\eq
\label{eq:ControlBregmann}
\mathscr D\leq 2{\lambda}\,.
\qe
Moreover, using Lemma \ref{lem:QIC}, note that:
\begin{align}
\mathscr D&=\lVert\hat{\mbf x}\lVert_{TV}-\int_{-1}^{1}q\mathrm d\hat{\mbf x}\,,\notag\\
&\geq\sum_{k=1}^{\hat s}\lvert\hat{a}_{k}\lvert(1-\lvert q({\hat{t}_{k}})\lvert)\,,\notag\\
&\geq\sum_{k=1}^{\hat s}|\hat{a}_{k}|\min\{2C_{1} m^{2}\min_{ t\in\mbf T}d(t, {\hat{t}_{k}})^{2};2c_{0}^{2}C_{1}\}\,,\label{eq:DivergenceMin}
\end{align}
where $c_{0}=2\pi\cdot 0.1649$ and $C_{1}=0.00424$ and the proof of (i) follows.

Now, let $t_j \in\mbf T$ and consider the polynomial $q_{t_j}$ described in Lemma \ref{lem:DualCertificate}. Using \eqref{eq:DivergenceMin} we get that:
\begin{align}
&|\sum_{\{k\ |\ d( t_j,\hat{ t}_{k})>\frac{c_{0}}{ m}\}}\hat{ a}_{k}q_{t_j}({\hat{ t}_{k}})+\sum_{\{k\ |\ d( t_j,\hat{ t}_{k})\leq \frac{c_{0}}{ m}\}}\hat{ a}_{k}(q_{t_j}({\hat{t}_{k}})-1)|\notag\\
&\leq\sum_{\{k\ |\ d(t_j,\hat{\mbf t}_{k})>\frac{c_{0}}{ m}\}}|\hat{a}_{k}||q_{t_j}|({\hat{t}_{k}})+\sum_{\{k\ |\ d(t_j,\hat{t}_{k})\leq \frac{c_{0}}{ m}\}}|\hat{a}_{k}||q_{t_j}-1|({\hat{t}_{k}})\,,\notag\\
&\leq\sum_{k=1}^{\hat s}|\hat{ a}_{k}|\min\{C_{2} m^{2}\min_{ t\in\mbf T}d( t, {\hat{t}_{k}})^{2};1-c_{0}^{2}C_{1}\}\,,\notag\\ 
&\leq C'\times\sum_{k=1}^{\hat s}|\hat{ a}_{k}|\min\{2C_{1} m^{2}\min_{ t\in\mbf T}d(t, {\hat{ t}_{k}})^{2};2c_{0}^{2}C_{1}\}\,,\notag\\
&\leq 2C'\lambda\,.\label{eq:DivergenceMax}
\end{align}
where $C_{2}=0.25$ and $C'=\max\{\frac{C_{2}}{2C_{1}};\frac{1-c_{0}^{2}C_{1}}{2c_{0}^{2}C_{1}}\}=109.36$. Invoking \eqref{eq:MajorationPrediction}, we deduce that for all $i=1,\ldots,s$,
\begin{align*}
|a_{i}-\hat{\mbf x}(\mbf t_{i}&+\mathscr B({c_{0}}/{ m}))|\leq|\int q_{ t_{i}}\mathrm d\mbf x - \int q_{t_{i}}\mathrm d\hat{\mbf x}\\
&+\sum_{\{k\ |\ d(\mbf t_{i},\hat{\mbf t}_{k})>\frac{c_{0}}{ m}\}}\hat{ a}_{k}q_{\mbf t_{i}}({\hat{\mbf t}_{k}})+\sum_{\{k\ |\ d( t_{i},\hat{t}_{k})\leq \frac{c_{0}}{ m}\}}\hat{a}_{k}(q_{t_{i}}({\hat{ t}_{k}})-1) |\,,\\
\leq&2(C'+1)\lambda\,,
\end{align*}
where $t_{i}+\mathscr B({c_{0}}/{ m})=\{t\ |\ d( t_{i},t)\leq {c_{0}}/{ m}\}$, proving (ii). Finally, observe that (iii) is a consequence of the aforementioned inequalities.
\end{proof}

\section{Non-uniform spline reconstruction}
\subsection{Notations}
In this section, we assume that $ d\geq0$. Observe that the frame investigated in this paper covers the recovery problem of a non-uniform spline of degree $ d$ from its projection onto ${\mathds R}_{ m- d-1}[X]$, the space of algebraic polynomials of degree at most $ m- d-1$. Indeed, consider an univariate spline $\mbf f$ of degree $ d$ over the knot sequence $\mbf T=\{-1,  t_{1},  t_{2},\ldots,  t_{s},1\}$, that is a continuously differentiable function $\mbf f$ of order ${ d-1}$ piecewise-defined by:
\eq\notag
\mbf f=\ind_{[-1,  t_{1})}\,\mbf P_{0}+\ind_{[  t_{1},  t_{2})}\,\mbf P_{1}+\ldots+\ind_{[  t_{s-1},  t_{s})}\,\mbf P_{s-1}+\ind_{[  t_{s},1]}\,\mbf P_{s}\,,
\qe
where $\mbf P_{k}$ belongs to $\mathds R_{ d}[X]$, and for all subset $E\subseteq[-1,1]$, $\ind_{E}(t)$ equals $1$ if $t$ belongs to $E$ and $0$ otherwise. Consider $\mbf f^{( d+1)}$, the $( d+1)$-th distributional derivative of $\mbf f$. We have :
\eq\notag
\mbf f^{( d+1)}=\sum_{k=1}^s (\mbf P_{k}^{( d)}-\mbf P^{( d)}_{k-1})\,\delta_{  t_k}\,,
\qe
where $\mbf P_{k}^{( d)}\in{\mathds R}$ is the $ d$-th derivative of $\mbf P_{k}$.

The next lemma links the moments of the spline $\mbf f$ to the ones of the signed measure $\mbf f ^{(d+1)}$.
\begin{lem}
\label{lem:Matrix}
\eq
\label{eq:MatrixFormulationAssumption1}
\mbf c(\mbf f^{(\mbf d+1)})=\left[\begin{array}{cc}0 & W_1 \\(-1)^{\mbf d+1}\,\mathrm{Id}_{\mbf m-\mbf d} & W_2\end{array}\right]\left(\begin{array}{c}\mbf p(\mbf f) \\\mbf b\end{array}\right)\,,
\qe
where:
\begin{itemize}
\item
 $\mbf p(\mbf f)=(\langle \mbf f, \varphi_{ d+1}^{( d+1)}\rangle,\langle \mbf f, \varphi_{ d+2}^{( d+1)}\rangle,\ldots,\langle \mbf f, \varphi_{ m}^{( d+1)}\rangle)$, 
 \item
 $\mbf b=(\mbf P_{0}(-1),\ldots,\mbf P^{( d-1)}_{0}(-1),\mbf P^{( d)}_{0}(-1),\mbf P_{s}(1),\ldots,\mbf P^{( d-1)}_{s}(1),\mbf P^{( d)}_{s}(1))$,
 \item and $W_{1},W_{2}$ are known matrices, defined by relations \eqref{eq:MediumMomentSplines}, \eqref{eq:LowMomentSplines} and \eqref{eq:FirstMomentSplines}, whose entries belong to the set $\{-1,1,\sqrt 2\,(-1)^{m}\mbf w_{k,l}\,;\ m\in\{0,1\}\ \mathrm{and}\ k,l\in\mathds N\}$ where $\mbf w_{k,l}$ are constants defined in \eqref{eq:wkl}.
\end{itemize}
\end{lem}
\begin{proof}
By induction, for $k=0,1,\ldots,{ m}$,
\eq\notag
c_{k}(\mbf f^{( d+1)})=\langle \mbf f^{( d+1)}, \varphi_{k}\rangle=\sum_{l=0}^{ d}(-1)^{l}\big[\mbf f^{( d-l)} \varphi_{k}^{(l)}\big]_{-1}^{1}+(-1)^{ d+1}\langle \mbf f, \varphi_{k}^{( d+1)}\rangle\,.\\
\qe
Moreover, it is known that for  all integers $k, l$, $T_{k}^{(l)}(-1)=(-1)^{k+l}\mbf w_{k,l}$ and $T_{k}^{(l)}(1)=\mbf w_{k,l}$ where:
\eq\label{eq:wkl}
\mbf w_{k,l}:= \ind_{\{ k \geq l\}} \prod_{j=0}^{l-1}\frac{k^{2}-j^{2}}{2j+1}\, .
\qe
Therefore, for $ m\geq k> d$,
\begin{align}\label{eq:MediumMomentSplines}
c_{k}(\mbf f^{( d+1)})=&\sqrt 2\,\sum_{l=0}^{ d}(-1)^{l}\mbf w_{k,l}\,\mbf P_{s}^{( d-l)}(1)\\&+(-1)^{k+1}\,\sqrt 2\,\sum_{l=0}^{ d}\mbf w_{k,l}\,\mbf P_{0}^{( d-l)}(-1)+(-1)^{ d+1}\langle \mbf f, \varphi_{k}^{( d+1)}\rangle\,,\notag
\end{align}
for $ d\geq k\geq1$,
\begin{align}
c_{k}(\mbf f^{( d+1)})=&\sqrt 2\,\sum_{l=0}^{k}(-1)^{l}\mbf w_{k,l}\,\mbf P_{s}^{( d-l)}(1)%\notag\\&
+(-1)^{k+1}\,\sqrt 2\,\sum_{l=0}^{k}\mbf w_{k,l}\,\mbf P_{0}^{( d-l)}(-1)\,,\label{eq:LowMomentSplines}
\end{align}
and 
\eq\label{eq:FirstMomentSplines}
c_{0}(\mbf f^{( d+1)})=\mbf P_{s}^{( d)}-\mbf P^{( d)}_{0}\,,
\qe
as claimed.
\end{proof}

\ni
\begin{rem}
The family $\{\varphi_{ d+1}^{( d+1)}, \varphi_{ d+2}^{( d+1)},\ldots, \varphi_{ m}^{( d+1)}\}$ is a basis of ${\mathds R}_{ m- d-1}[X]$, so $\mbf p(\mbf f)$ is entirely determined by any projection of $\mbf f$ onto ${\mathds R}_{ m- d-1}[X]$. %(up to a change of basis). 
\end{rem}

\begin{rem}
Observe that $\mbf b$ describes the boundary conditions on $\mbf f$. Recall that, in our model, we assume that the experimenter knows these boundary conditions. Furthermore, Equation \eqref{eq:MatrixFormulationAssumption1} shows that the noiseless moments appearing in \eqref{def:BlassoSpline} can be determined by the boundary conditions $\mbf b$.% which are supposed to be known. %Eventually, note that the $\mbf c(\mbf f^{( d+1)})$ can be obtained from the projection $\mbf p(\mbf f)$ and $\mbf b$. Then, our model considers a gaussian perturbation of the projection $\mbf p(\mbf f) $.
\end{rem}

\subsection{Observation of a random perturbation}

\begin{Assumption}[Approximate projection of non-uniform splines]\label{eq:AssumptionSplines}
We say that a random polynomial $P$ with values in ${\mathds R}_{ m- d-1}[X]$ satisfies {\bf Assumption \ref{eq:AssumptionSplines}} if
\eq\notag
\Theta(P)\sim\mathcal N(\mbf p(\mbf f),\sigma^{2}\,\mathrm{Id}_{ m- d})\,,
\qe
where $\Theta(P):= (\langle P, \varphi_{ d+1}^{( d+1)}\rangle,\langle P, \varphi_{ d+2}^{( d+1)}\rangle,\ldots,\langle P, \varphi_{ m}^{( d+1)}\rangle)$.
\end{Assumption}
%\begin{rem}
%For the sake of simplicity, we choose a Gaussian approximation error but our analysis can be extended to other types of noise. This extension can be done by bounding the $\ell_{\infty}$-norm of the polynomial whose coefficients are given by the random vector $\mbf e=(0,\mbf n)$ , as done in Lemma \ref{lem:Rice}. 
%\end{rem}

\begin{rem}
Note that Assumption \ref{eq:AssumptionSplines} asserts that the experimenter observes a Gaussian perturbation (with known covariance matrix) of the inner-products of the non-uniform spline $\mbf f$ with the polynomial basis $\{\varphi_{ d+1}^{( d+1)}, \varphi_{ d+2}^{( d+1)},\ldots, \varphi_{ m}^{( d+1)}\}$. In particular, observe that $\lVert\varphi_{ m}^{( d+1)}\lVert_2^2=\mathcal O[(\frac{m!}{(m-d-1)!})^2]$ so that the signal-to-noise ratio (SNR) is of the order of $(\frac{m!}{\sigma(m-d-1)!})^2$. In applications, the standard assumption is that the SNR depends only on the noise variance. To match this situation, one needs to consider a noise level $\sigma:=\sigma_0\frac{m!}{(m-d-1)!}$ in order to get a SNR of the order of $1/\sigma_0^2$. For sake of readability, we do not pursue on this idea but the simulations of this paper are made accordingly.
\end{rem}

%\begin{rem}
%Observe the mapping $\Theta$ defines an isomorphism from ${\mathds R}_{ m- d-1}[X]$ onto $\mathds R^{ m- d}$. Moreover, note that the inverse image of $\mbf n\sim\mathcal N(0,\sigma^{2}\,\mathrm{Id}_{ m- d})$ under $\Theta$ is a random polynomial whose entries are Gaussian random variables (whose covariance matrix can be computed and depends on the basis of ${\mathds R}_{ m- d-1}[X]$ the experimenter is considering). Therefore, Assumption \ref{eq:AssumptionSplines} can be equivalently formulated as the knowledge of some Gaussian perturbation of some projection $\mbf p(\mbf f)$ in a given basis of ${\mathds R}_{ m- d-1}[X]$.
%\end{rem}

\begin{rem}
Remark that the noisy moments appearing in \eqref{def:BlassoSpline} are a Gaussian perturbation of the moments described by \eqref{eq:MatrixFormulationAssumption1}.
\end{rem}

\subsection{Algorithm and main theorem}
\ni
Let $P$ be a random vector with values in ${\mathds R}_{ m- d-1}[X]$. Set:
\eq
\label{def:BlassoSpline}
 \hat{\mbf x}\in\arg\min_{ \mu\in\mbf C_{ d}(\mbf f^{( d+1)})}\frac12\lVert \mbf c(\mu)-\mbf y\lVert^2_2+\lambda\lVert\mu\lVert_{TV}\,.
\qe
Recall that $\mbf C_{ d}(\mbf f^{( d+1)}):=\{\mu\in \mathcal M\,;\quad \forall\,k=0,\ldots, d\,,\ c_{k}(\mu)=c_{k}(\mbf f^{( d+1)})\}$, $\lambda>0$ is a tuning parameter and 
\[
\mbf y:=\left[\begin{array}{cc}0 & W_1 \\(-1)^{ d+1}\,\mathrm{Id}_{ m- d} & W_2\end{array}\right]\left(\begin{array}{c}\Theta(P) \\\mbf b\end{array}\right)\,.
\]
\begin{rem}
Note that if a discrete measure $\hat{\mbf x}$ enjoys
\eq
\label{eq:ContrainteSpline}
\forall k=0,\ldots, d,\quad
c_{k}(\hat{\mbf x})=c_{k}(\mbf f^{( d+1)})
\qe
then one can explicitly construct the unique non-uniform spline $\hat{\mbf f}$ with $( d+1)$-th derivative $\hat{\mbf x}$ and boundary conditions $\mbf b$. Indeed, observe that we can uniquely construct a non-uniform spline $\hat{\mbf f}$ from the knowledge of the $( d+1)$ boundary conditions at point $-1$ and its $( d+1)$-th derivative. Moreover, Eq.'s \eqref{eq:ContrainteSpline}, \eqref{eq:LowMomentSplines} and \eqref{eq:FirstMomentSplines} show that $\hat{\mbf f}$ satisfies the $( d+1)$ boundary conditions at point $1$ and so the boundary conditions $\mbf b$.
\end{rem}
%Eventually, we consider 

\begin{algorithm}%[htbp]
 \caption{Non-uniform spline recovery algorithm}%
 \label{alg:NUSR}%
{%
\KwIn{Boundary conditions $\mbf b$, a polynomial approximation $P$, an upper bound $\sigma$ on the noise standard deviation and $\alpha>0$ a tuning parameter.}
\KwOut{A non-uniform spline $\hat{\mbf f}$.}
\begin{enumerate}
\item Set $ d=\mathrm{Size}(\mbf b)/2-1$ and $ m=\mathrm{deg}(P)+ d+1$,
\item  Compute $\Theta(P)=(\langle P, \varphi_{ d+1}^{( d+1)}\rangle,\langle P, \varphi_{ d+2}^{( d+1)}\rangle,\ldots,\langle P, \varphi_{ m}^{( d+1)}\rangle)$,
\item Compute 
$
\mbf y=\left[\begin{array}{cc}0 & W_1 \\(-1)^{ d+1}\,\mathrm{Id}_{ m- d} & W_2\end{array}\right]\left(\begin{array}{c}\Theta(P) \\\mbf b\end{array}\right)\,,
$\\
where $W_{1}$ and $W_{2}$ are described in Lemma \ref{lem:Matrix}.
\item  Set $\displaystyle
\lambda=4\sigma[2(1+\alpha)( m- d)\log(5( m+ d+1))]^{1/2}$,
\item Find a discrete solution $\displaystyle
\hat{\mbf x}=\sum_{k=1}^{\hat s}\hat{ a}_{k}\delta_{\hat{\mbf t}_{k}}$ to \eqref{def:BlassoSpline}\\ using SDP programming, see Appendix \ref{app:Back},
\item Find the unique spline $\hat{\mbf f}$ of order $ d-1$ such that $\displaystyle
\hat{\mbf f}^{( d+1)}=\hat{\mbf x}$ and $\displaystyle
(\hat{\mbf f}_{0}(-1),\ldots,\hat{\mbf f}^{( d-1)}_{0}(-1),\hat{\mbf f}^{( d)}_{0},\hat{\mbf f}_{s}(1),\ldots,\hat{\mbf f}^{( d-1)}_{s}(1),\hat{\mbf f}^{( d)}_{s})=\mbf b$.
\end{enumerate}
}%
\end{algorithm}

\begin{thm}\label{thm:Spline}
Let $ m> d\geq0$. Let $\mbf f$ be a non-uniform spline of degree $ d$ that can be written as:
\eq\notag
\mbf f=\ind_{[-1,  t_{1})}\,\mbf P_{0}+\ind_{[  t_{1},  t_{2})}\,\mbf P_{1}+\ldots+\ind_{[  t_{s-1},  t_{s})}\,\mbf P_{s-1}+\ind_{[  t_{s},1]}\,\mbf P_{s}\,,
\qe
where $\mbf P_{k}\in\mathds R_{ d}[X]$ and $\mbf T=\{-1,  t_{1},  t_{2},\ldots,  t_{s},1\}$ enjoys:
\eq\notag
\min\{\Delta (\mbf T),\,2\epsilon(\mbf T)\} \geq \frac{5\pi}{ m}.
\qe
Set $\mbf b=(\mbf P_{0}(-1),\ldots,\mbf P^{( d-1)}_{0}(-1),\mbf P^{( d)}_{0},\mbf P_{s}(1),\ldots,\mbf P^{( d-1)}_{s}(1),\mbf P^{( d)}_{s})$ and let $P$ be such that {\bf Assumption 1} holds. Let $\alpha>0$ then, with probability greater than $1-\big[\frac1{5( m+ d)}\big]^{\alpha}$, any output $\hat{\mbf f}$ of Algorithm \ref{alg:NUSR} enjoys:
\begin{enumerate}
\item Global control:
\[
\displaystyle\sum_{k=1}^{\hat s}|\hat{\mbf P}_{k}^{( d)}-\hat{\mbf P}^{( d)}_{k-1}|\min\Big\{ m^{2}\min_{ t\in\mbf T}d( t, {\hat{ t}_{k}})^{2};c_{0}^{2}\Big\}\leq{c_{1}}\lambda\,,
\]
\item Large discontinuity localization: $\forall i=1,\ldots,s,\ \mathrm{s.t.}\ |\mbf P_{i}^{( d)}-\mbf P^{( d)}_{i-1}|>c_{2}\lambda$,
\[
\exists\,\hat{t}\in\{\hat{t}_{1},\ldots,\hat{t}_{\hat s}\}\ \mathrm{s.t.}\quad \displaystyle d({t}_{i},\hat{t})\leq \left[\frac{c_{1}\lambda}{|\mbf P_{i}^{( d)}-\mbf P^{( d)}_{i-1}|-c_{2}\lambda}\right]^{1/2}\frac1{ m}\,,
\]
\end{enumerate}
where  $c_{0}=1.0361$, $c_{1}=235.85$, $c_{2}=220.72$, $\lambda=4\sigma[2(1+\alpha)( m- d)\log(5( m+ d+1))]^{1/2}$ and $\hat{\mbf f}$ is written as:
\[
\notag
\hat{\mbf f}=\ind_{[-1,\hat{  t}_{1})}\,\hat{\mbf P}_{0}+\ind_{[\hat{  t}_{1},\hat{  t}_{2})}\,\hat{\mbf P}_{1}+\ldots+\ind_{[\hat{  t}_{\hat s-1},\hat{  t}_{\hat s})}\,\mbf P_{s-1}+\ind_{[\hat{  t}_{\hat s},1]}\,\mbf P_{\hat s}\,,
\]
with $\hat{\mbf P}_{k}\in\mathds R_{ d}[X]$.
\end{thm}
\begin{proof}
From \eqref{eq:MatrixFormulationAssumption1} deduce that if $P$ satisfies {\bf Assumption 1} then:
\eq
\notag
%\label{eq:ObservationSplines}
\mbf y:=\left[\begin{array}{cc}0 & W_1 \\(-1)^{ d+1}\,\mathrm{Id}_{ m- d} & W_2\end{array}\right]\left(\begin{array}{c}\Theta(P) \\\mbf b\end{array}\right)=\mbf c(\mbf f^{( d+1)})+(-1)^{ d+1}\,\left(\begin{array}{c}0 \\\mbf n\end{array}\right)\,,
\qe
where $W_{1}$ and $W_{2}$ are described in Lemma \ref{lem:Matrix}. Observe the result follows from Theorem \ref{thm:Main}.
\end{proof}

\section{Numerical experiments}

The semidefinite formulation of our procedure follows from standard arguments in super-resolution theory, see Appendix \ref{app:Fenchel} and Appendix \ref{app:Back}.

We have run several numerical experiments and we have observed the following behaviour. In most cases, our approach succeeds in localizing the knots of the original spline and the amplitudes of its discontinuties while some small discontinuities may appear in the reconstructed spline. 

Observe that, as can be seen in the second example of Figure \ref{fig:Example2}, a small error in the estimation of the amplitude of a  discontinuity may have a large impact on the reconstructed spline. More precisely, the $\ell_{\infty}$-distance between the orginial and reconstructed splines can be large. However, large discontinuities are well estimated (as proven in Theorem \ref{thm:Spline}) so that the overall profile of the original spline is well depicted by the reconstructed spline.

Finally \ref{fig:Example3} and \ref{fig:Example4} show on an example the behaviour of our algorithm when increasing the noise level $\sigma$, and with degrees $d$ higher than $1$.

\begin{figure}%[h!]
\begin{center}
\includegraphics[width=0.49\textwidth]{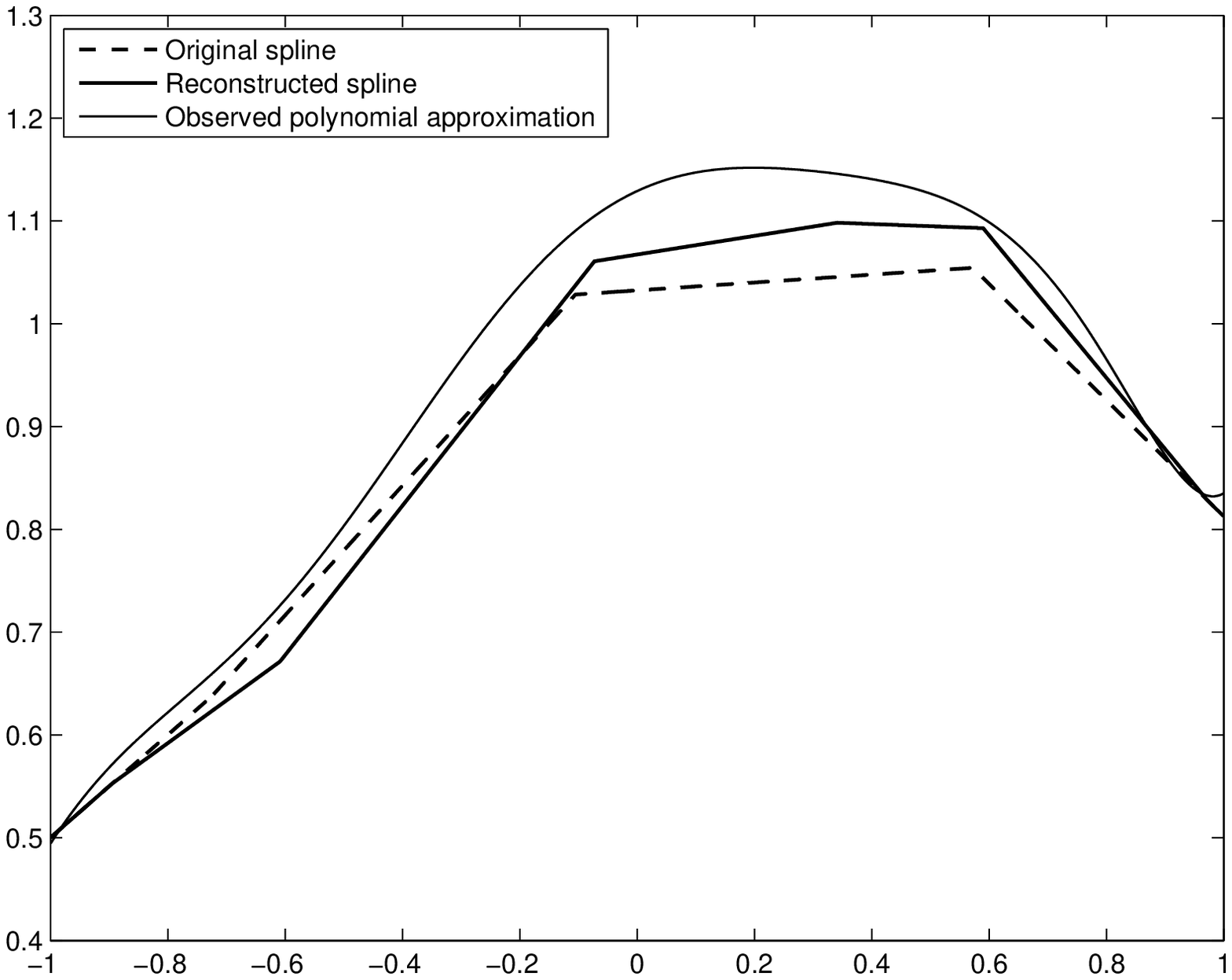}\includegraphics[width=0.49\textwidth]{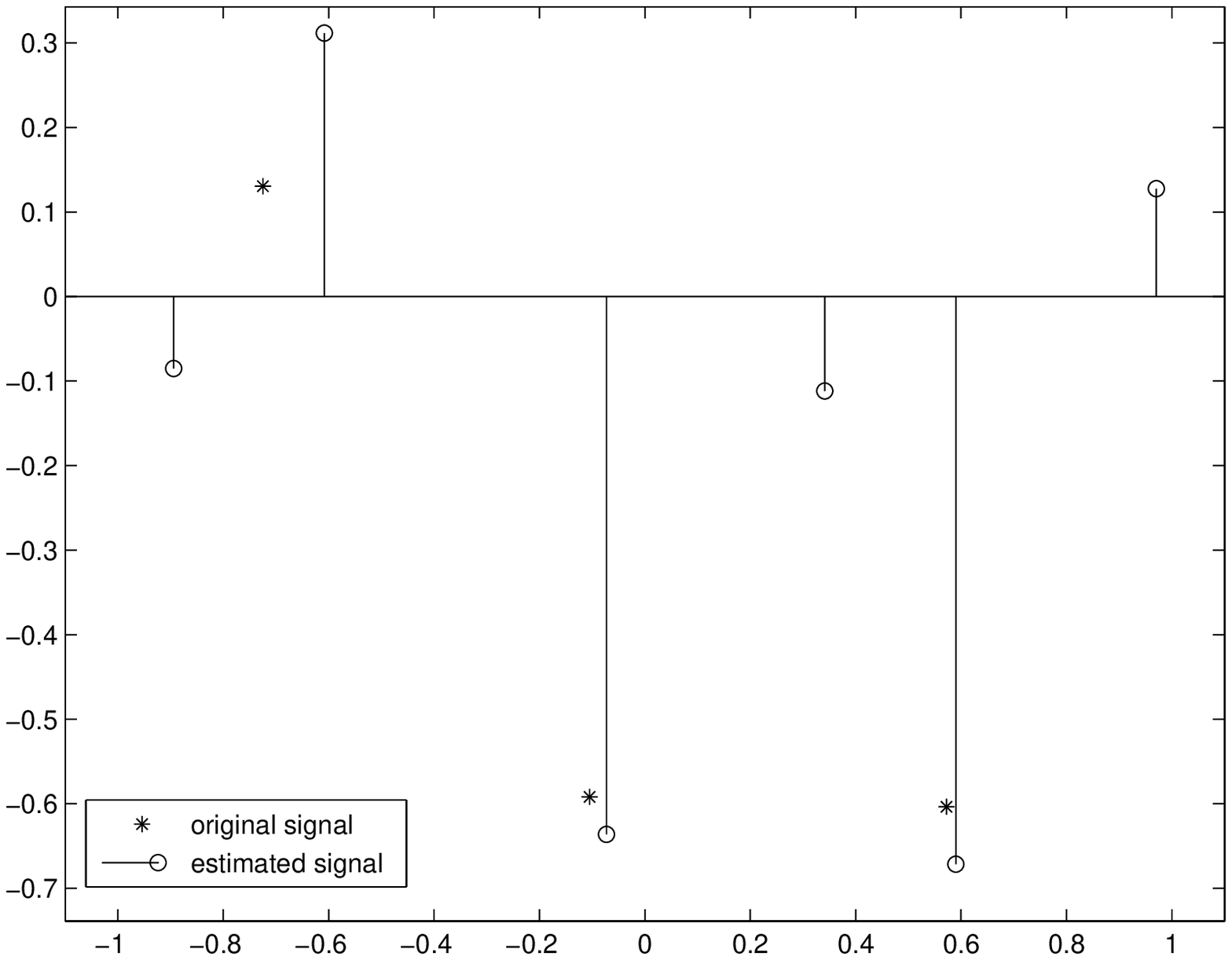}\\
\includegraphics[width=0.49\textwidth]{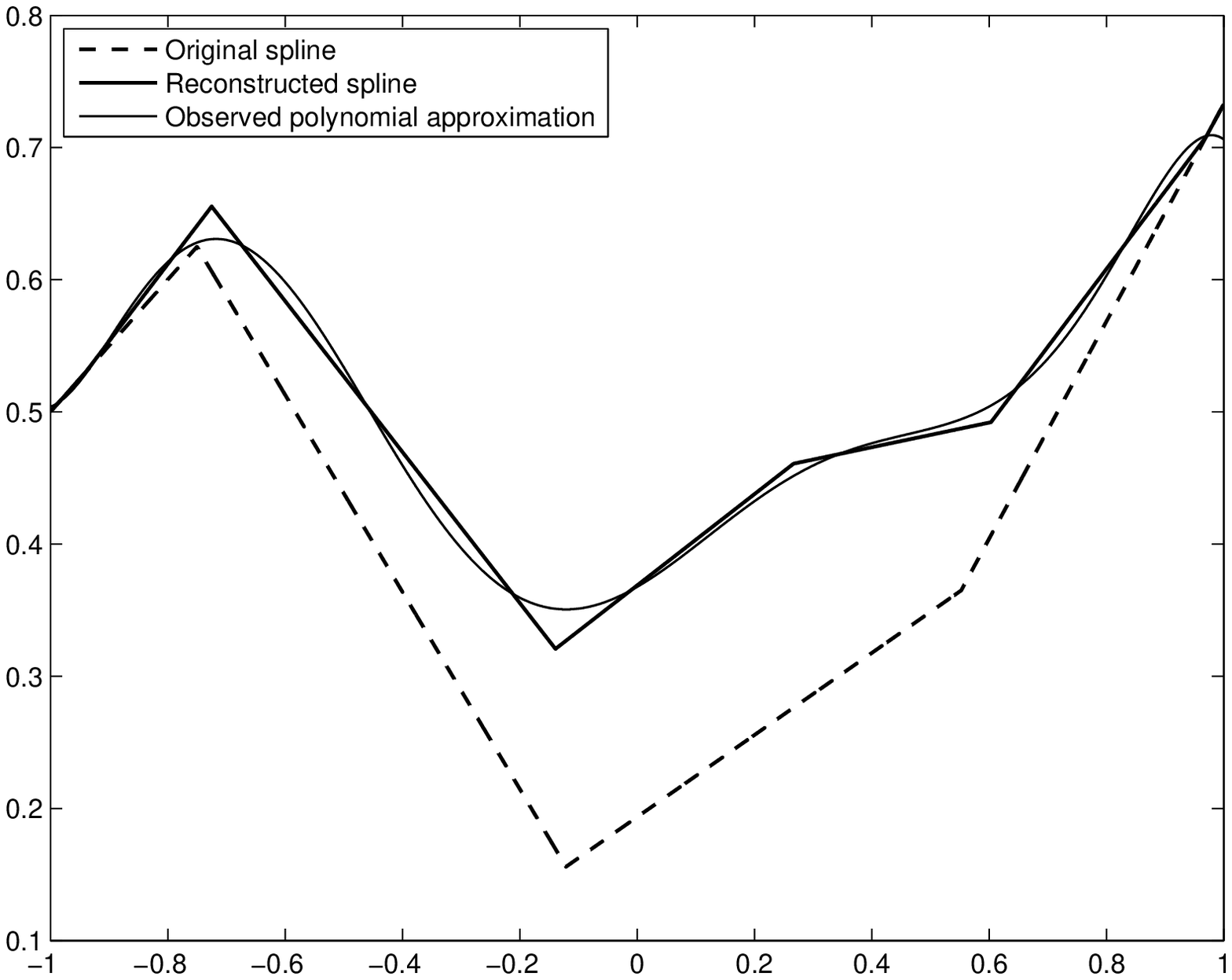}\includegraphics[width=0.49\textwidth]{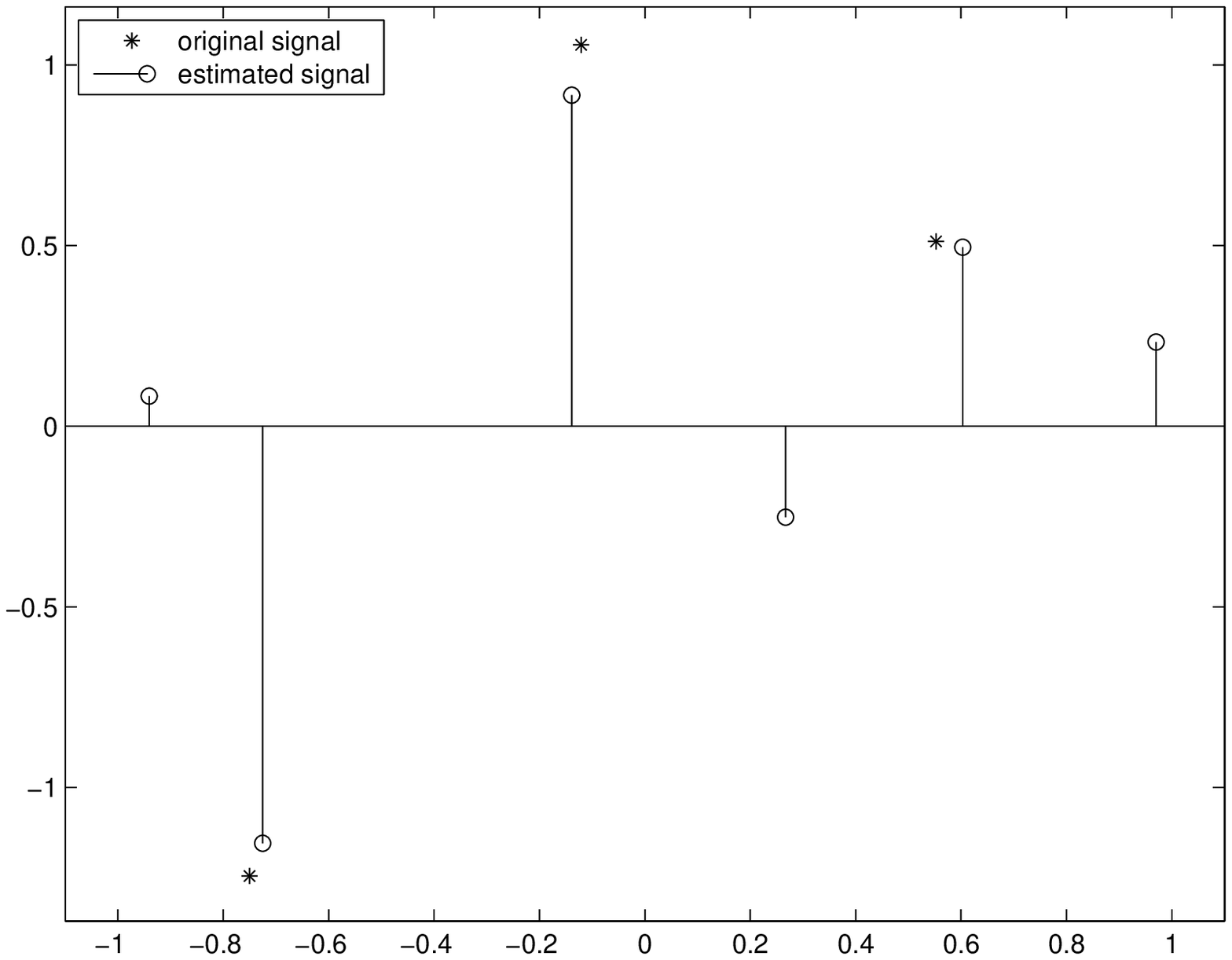}
\end{center}
\caption{Left : estimated spline (thick black line) of a non-uniform spline $\mbf f$ (thick dashed gray line) and its knots from a polynomial approximation (thin black line). Right : $d+1$-derivative of the spline (stars: original spline; circles: reconstructed spline).}
\label{fig:Example2}
\end{figure}

\begin{figure}%[h!]
\begin{center}
\includegraphics[width=0.33\textwidth]{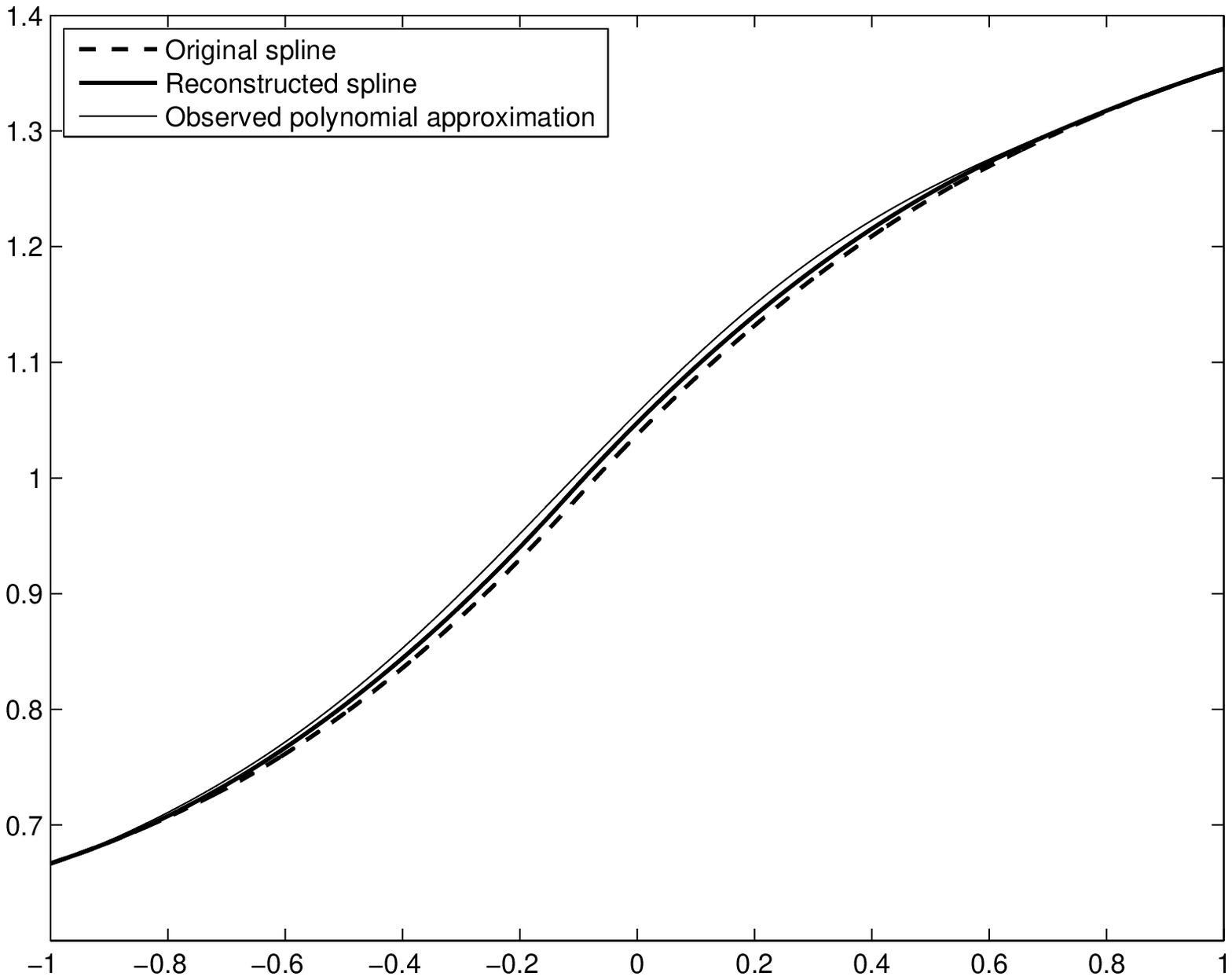}\includegraphics[width=0.33\textwidth]{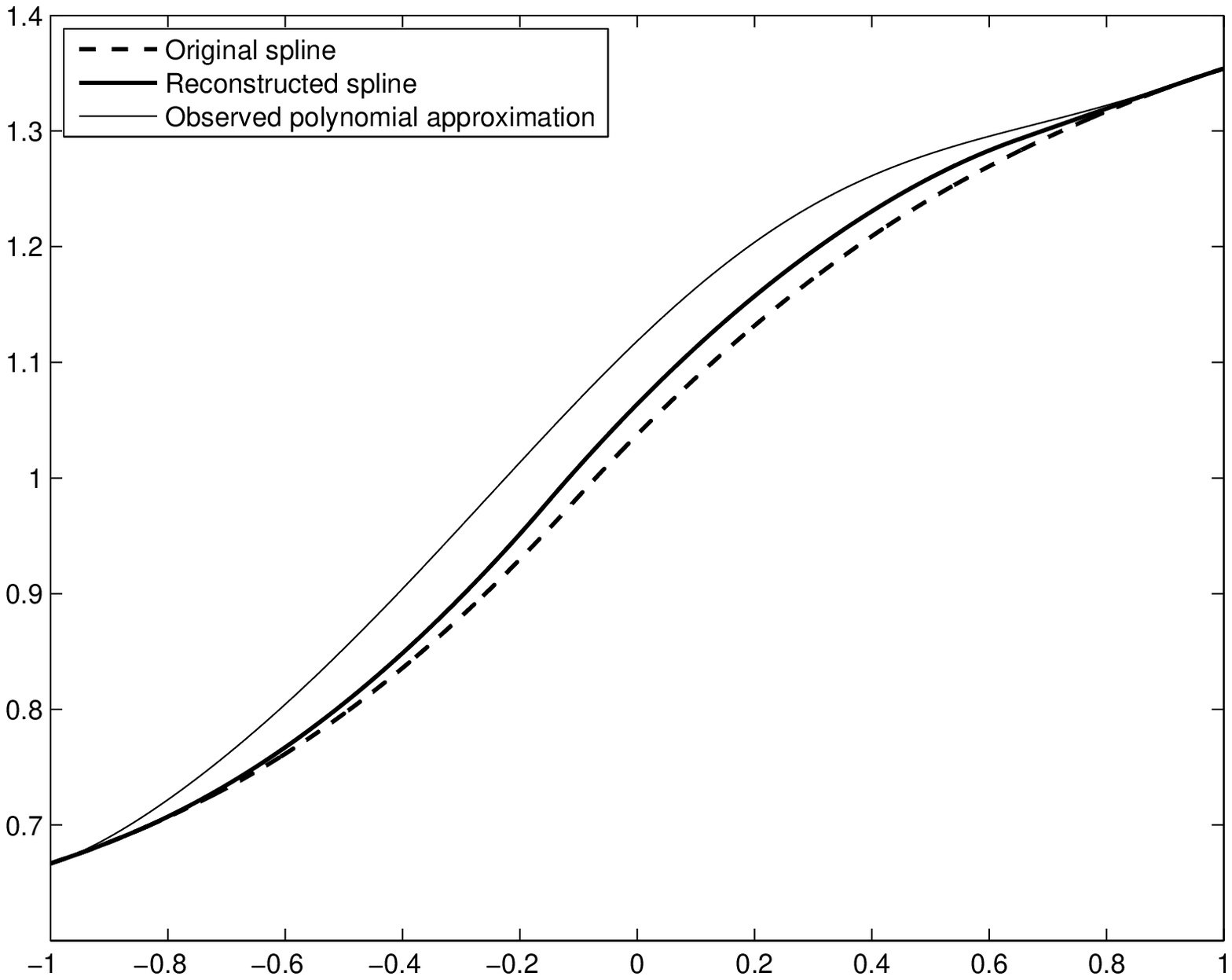}\includegraphics[width=0.33\textwidth]{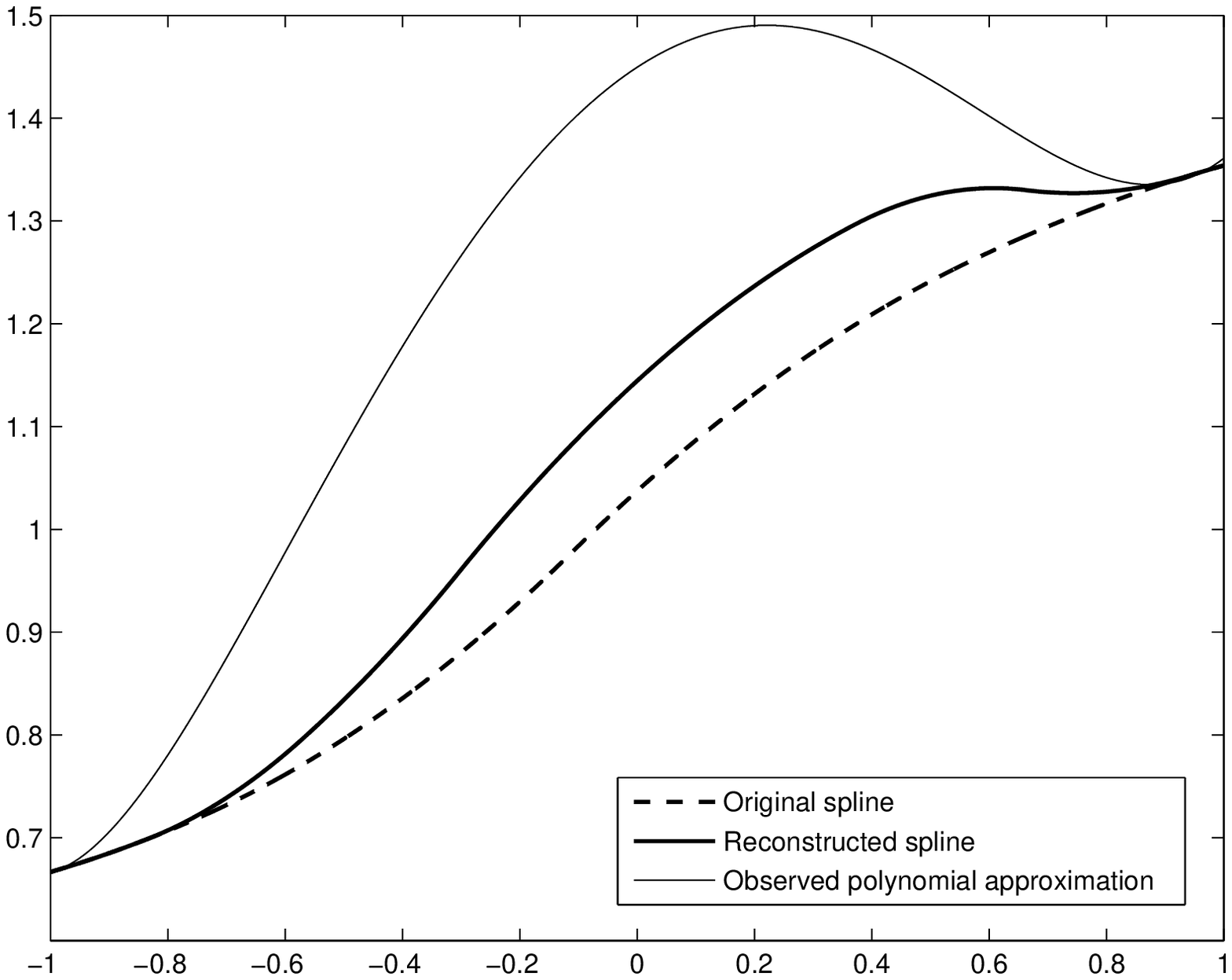}\\
\includegraphics[width=0.33\textwidth]{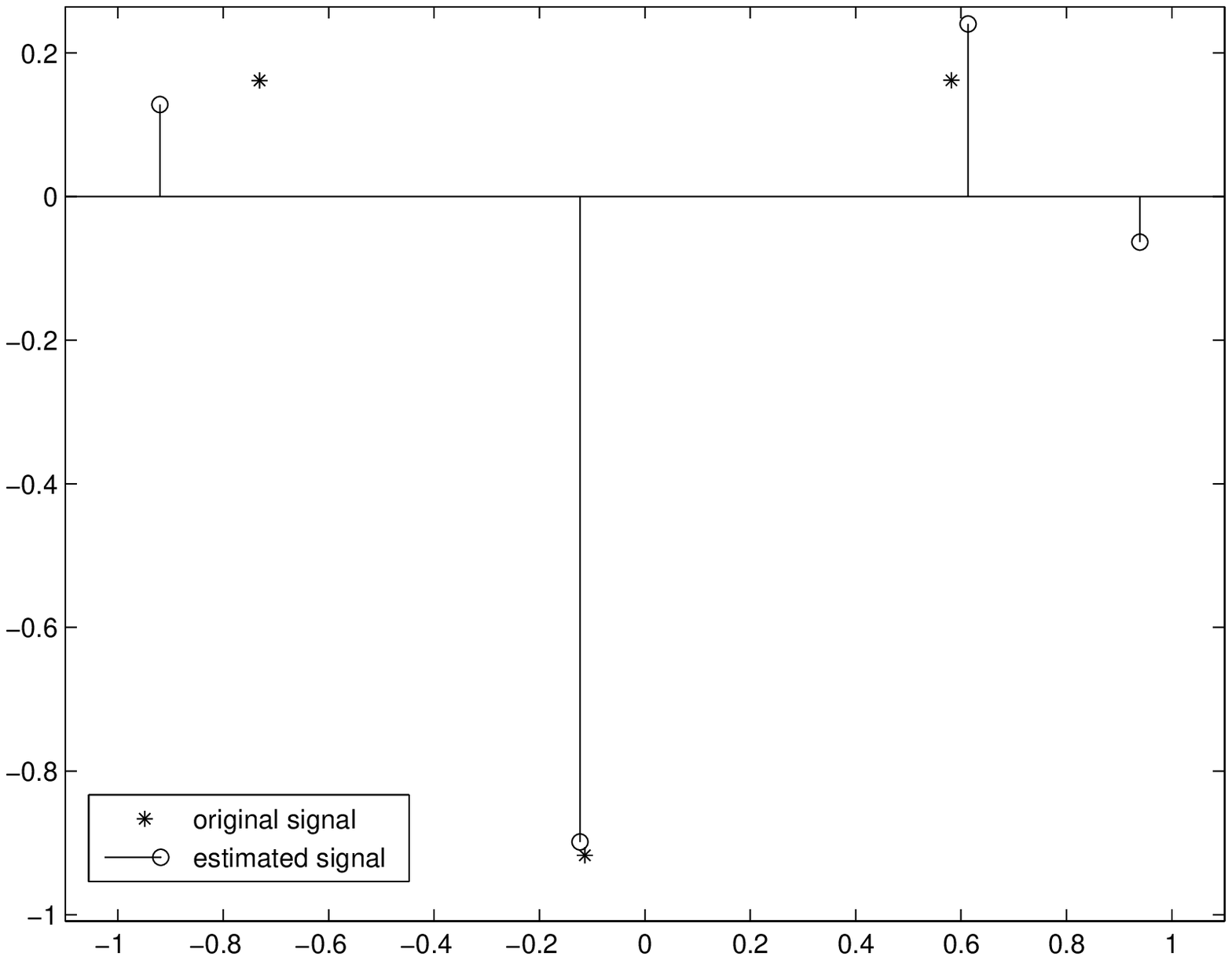}\includegraphics[width=0.33\textwidth]{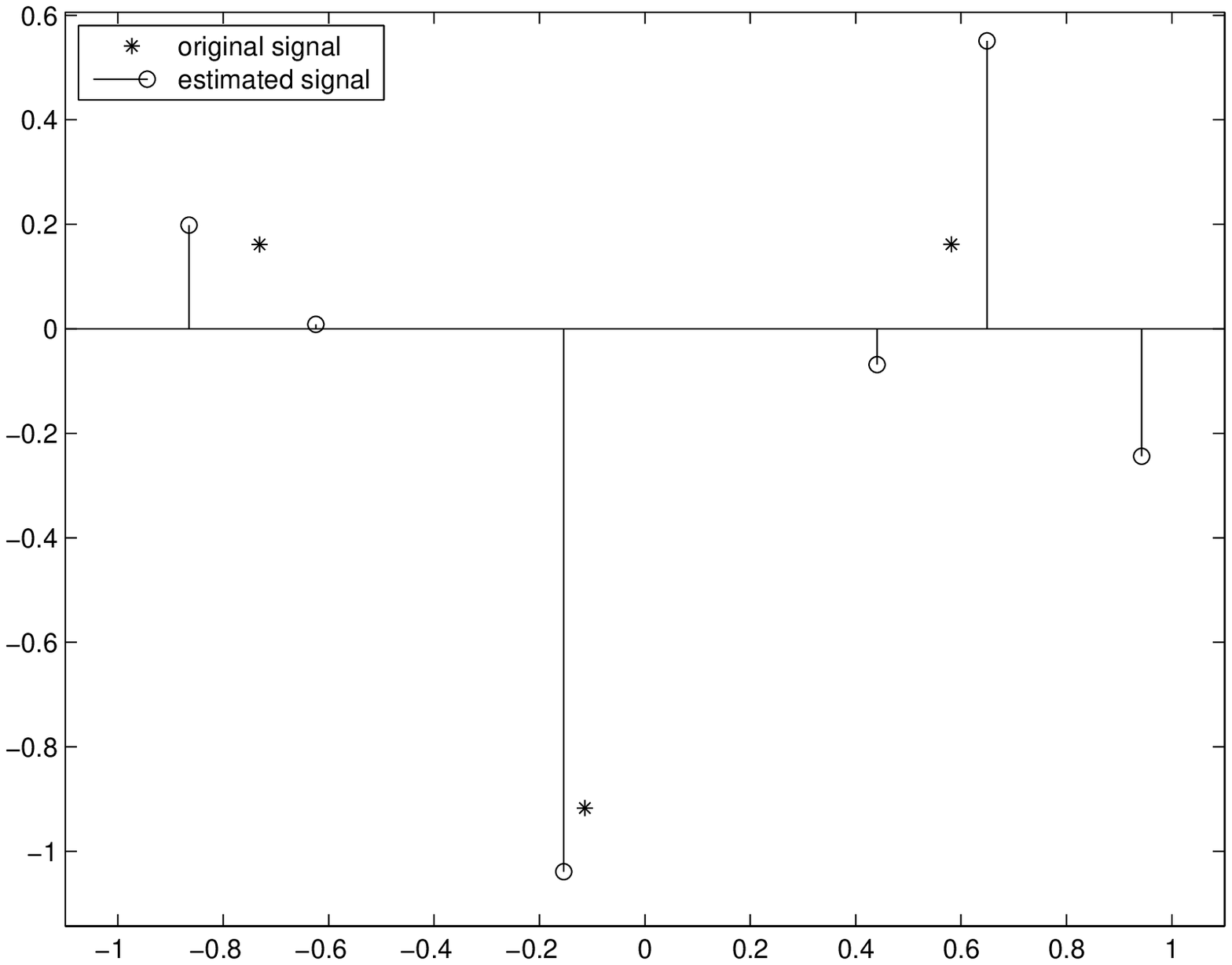}\includegraphics[width=0.33\textwidth]{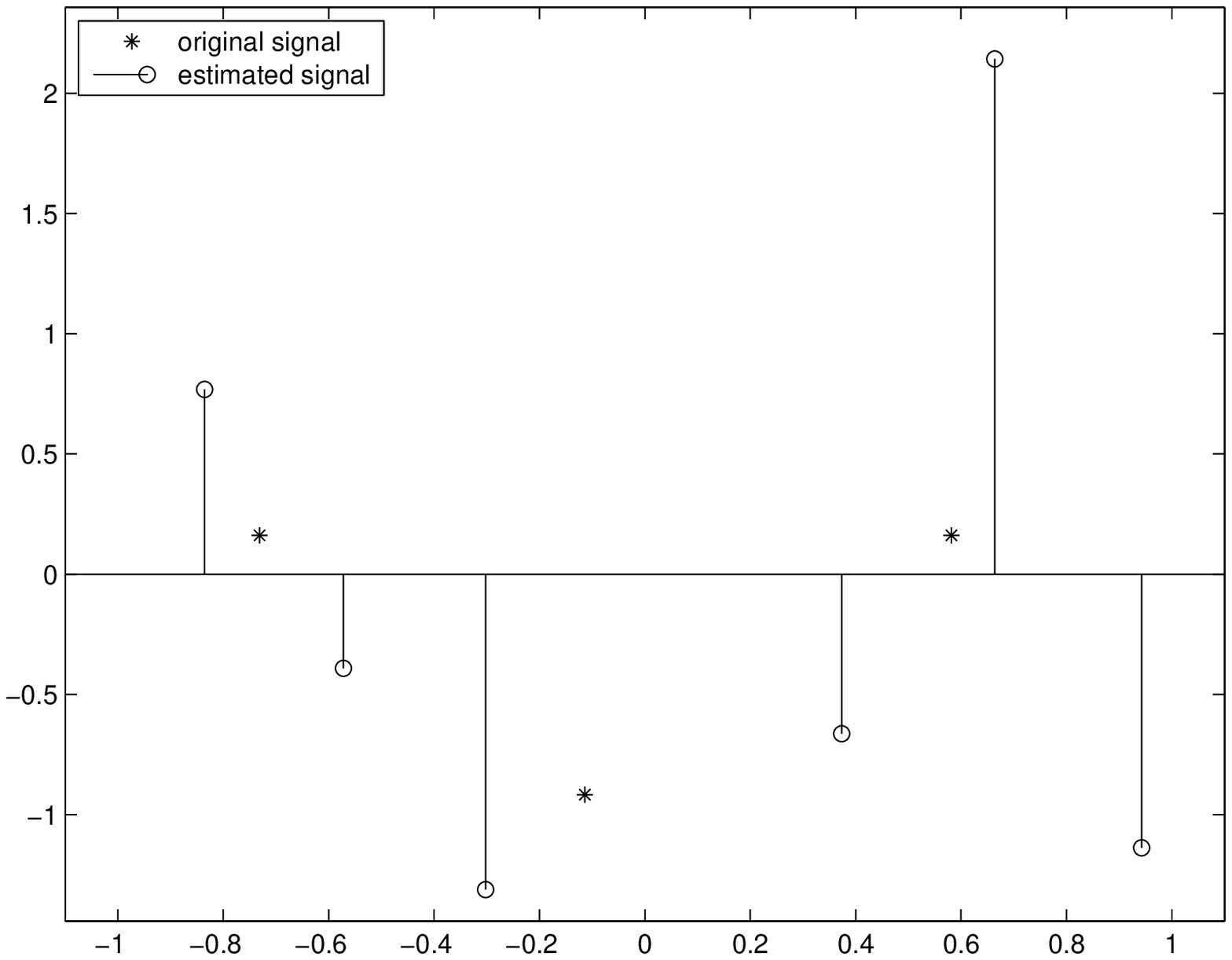}
\end{center}
\caption{Top : estimated spline (thick black line) of a non-uniform spline $\mbf f$ (thick dashed gray line) and its knots from a polynomial approximation (thin black line). Bottom : corresponding $d+1$-derivative of the spline (stars: original spline; circles: reconstructed spline). Degree $d = 2$, number of observed noisy moments $m - d= 8$. Noise levels $\sigma = \sigma_0 \frac{m!}{(m-d-1)!}$ with $\sigma_0 \equiv 0.0005, 0.002, 0.01$ (from left to right).}
\label{fig:Example3}
\end{figure}

\begin{figure}%[h!]
\begin{center}
\includegraphics[width=0.33\textwidth]{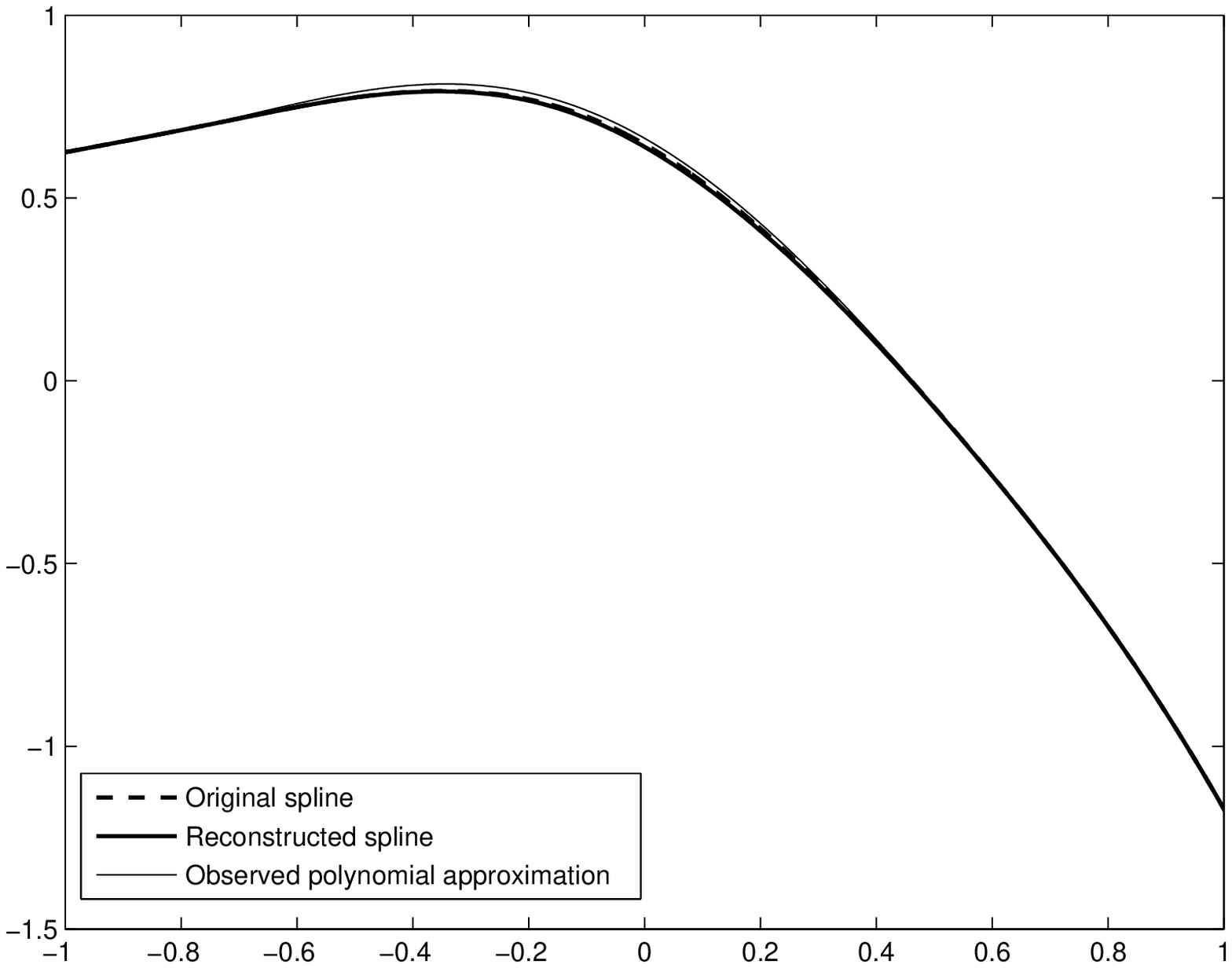}\includegraphics[width=0.33\textwidth]{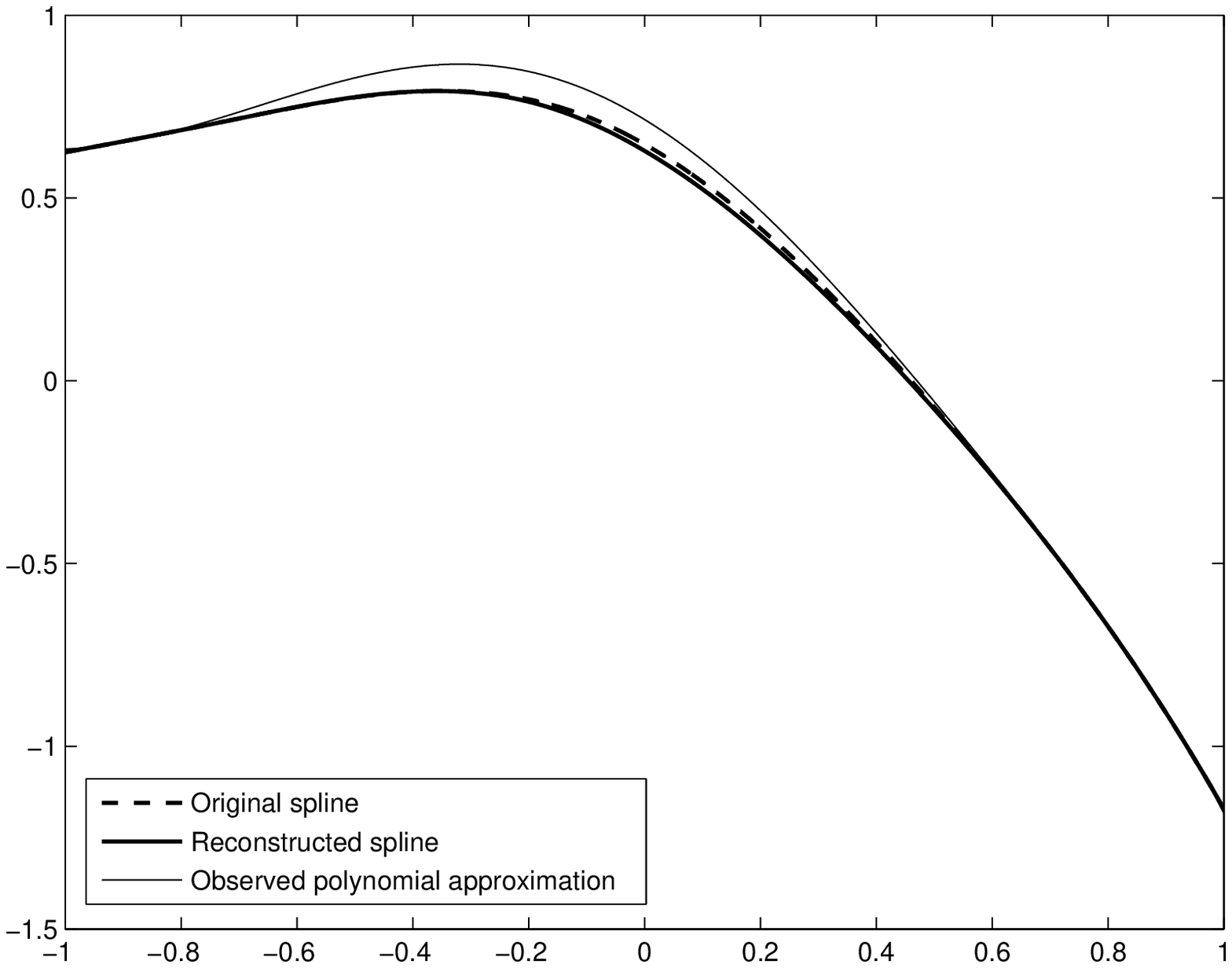}\includegraphics[width=0.33\textwidth]{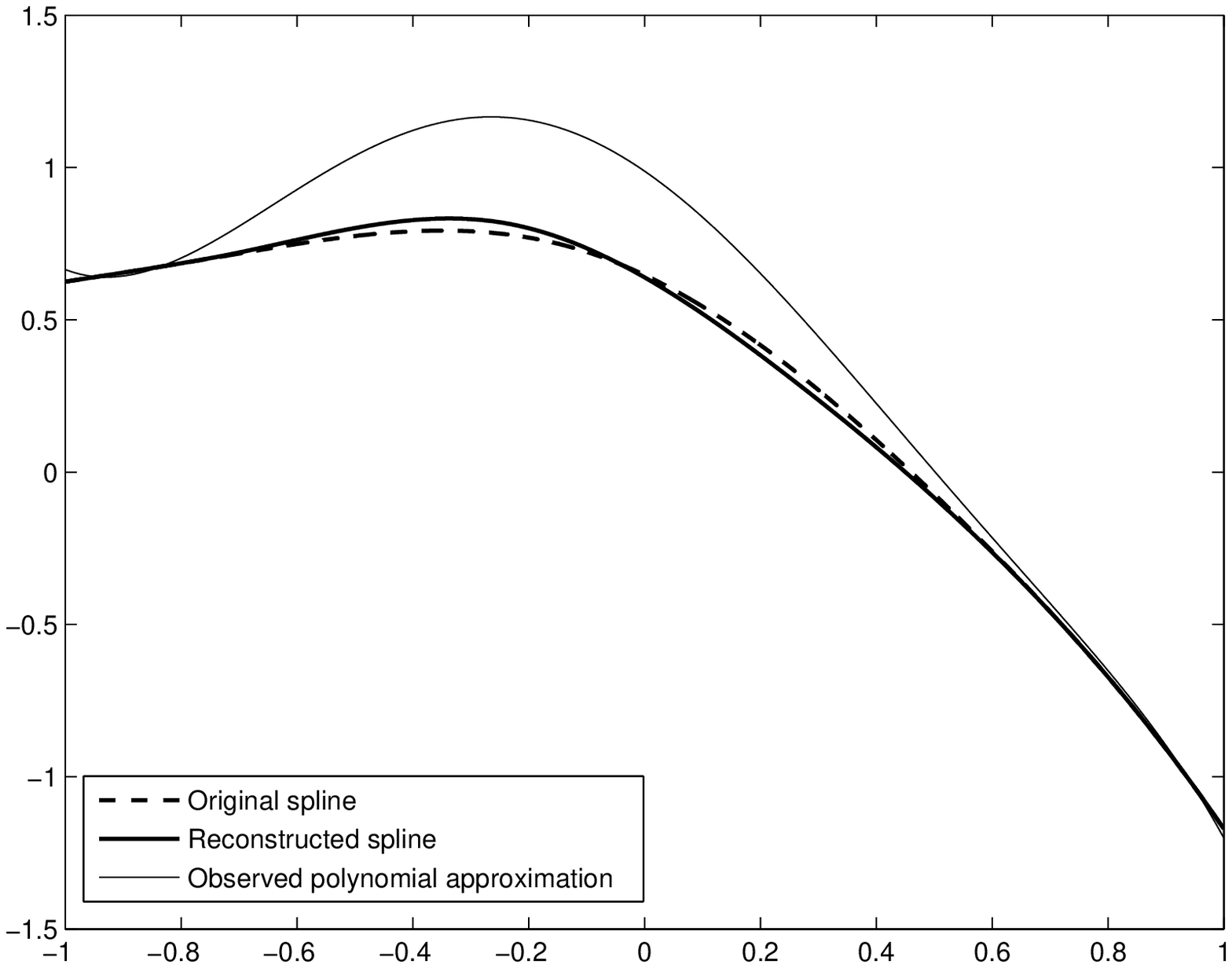}\\
\includegraphics[width=0.33\textwidth]{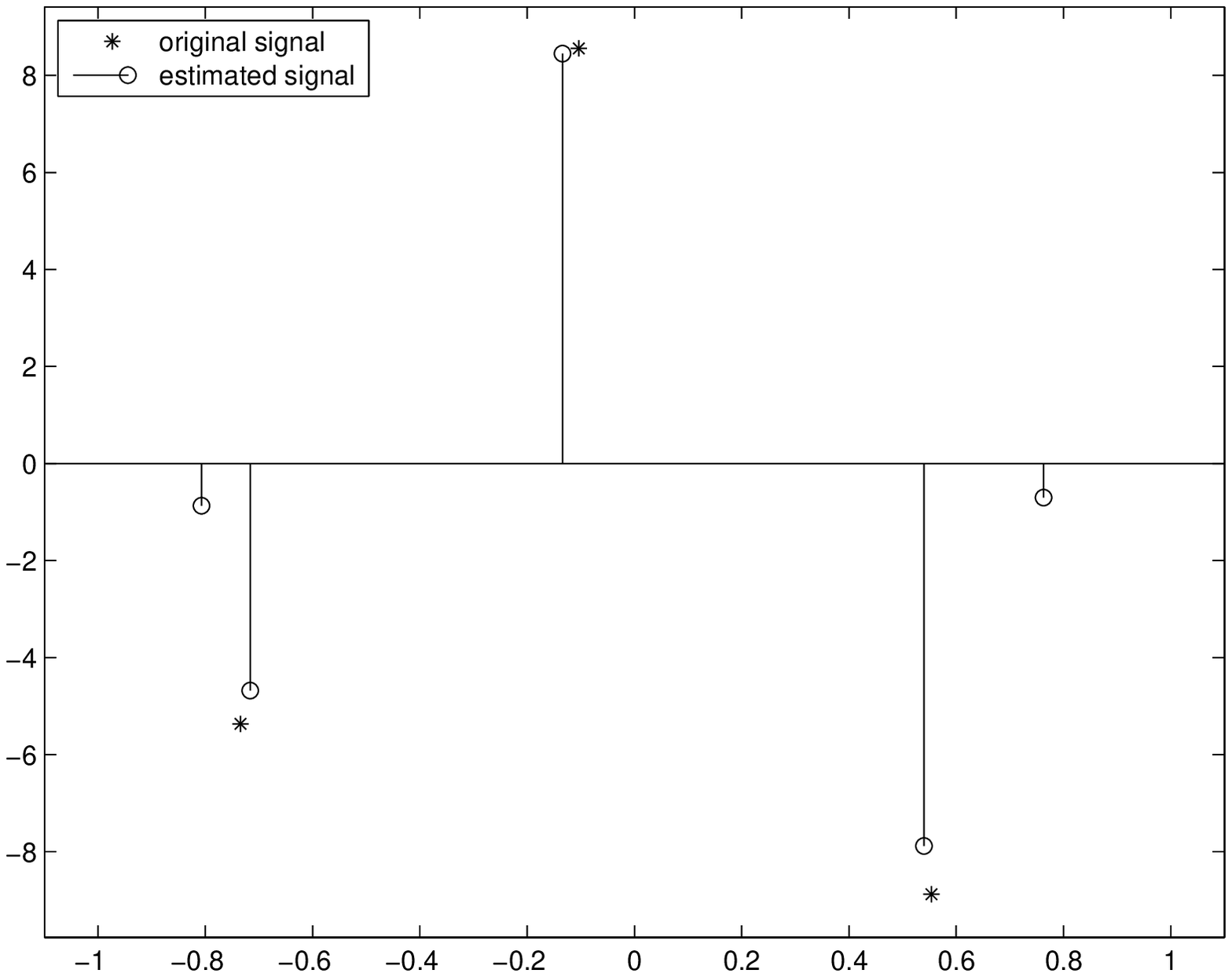}\includegraphics[width=0.33\textwidth]{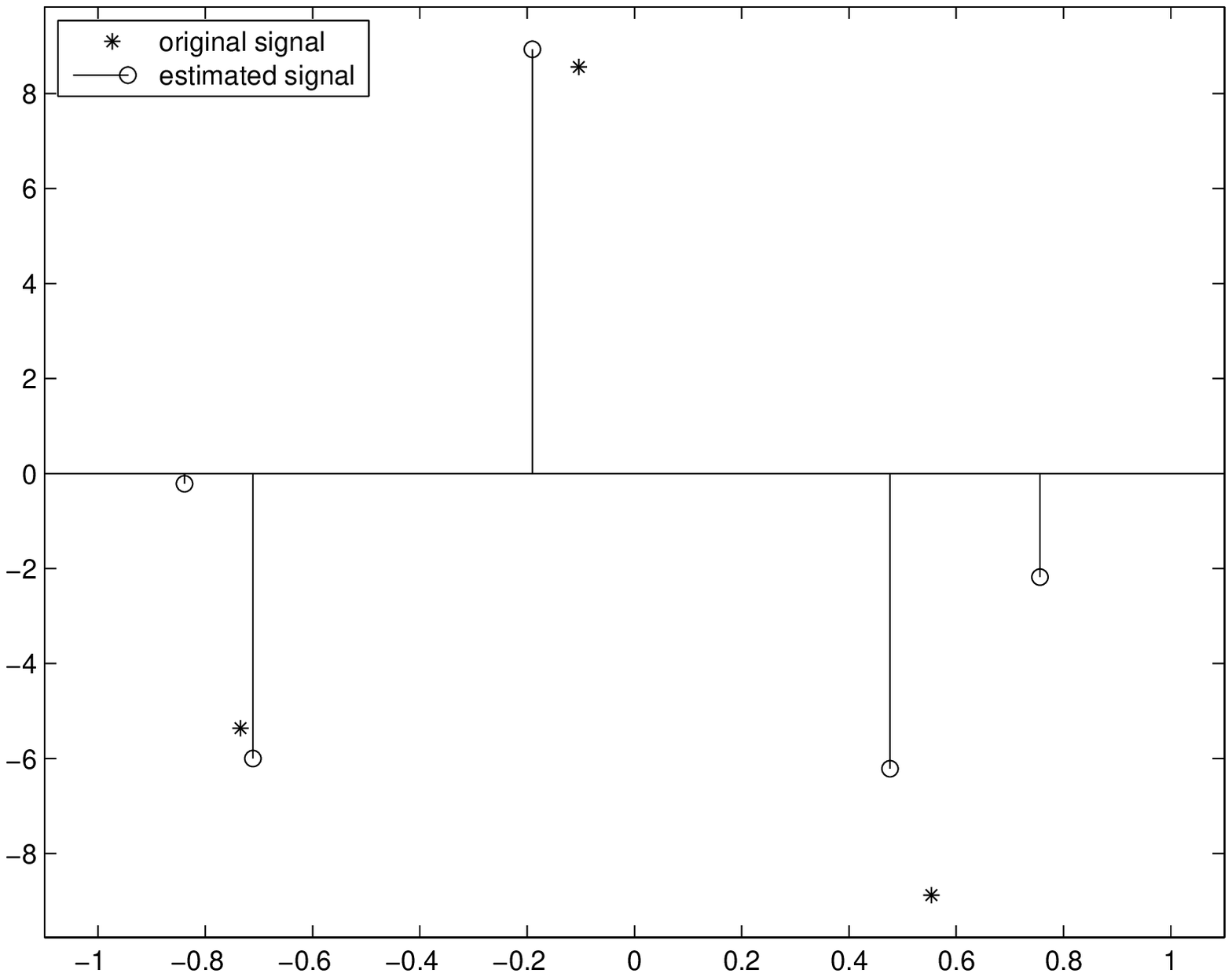}\includegraphics[width=0.33\textwidth]{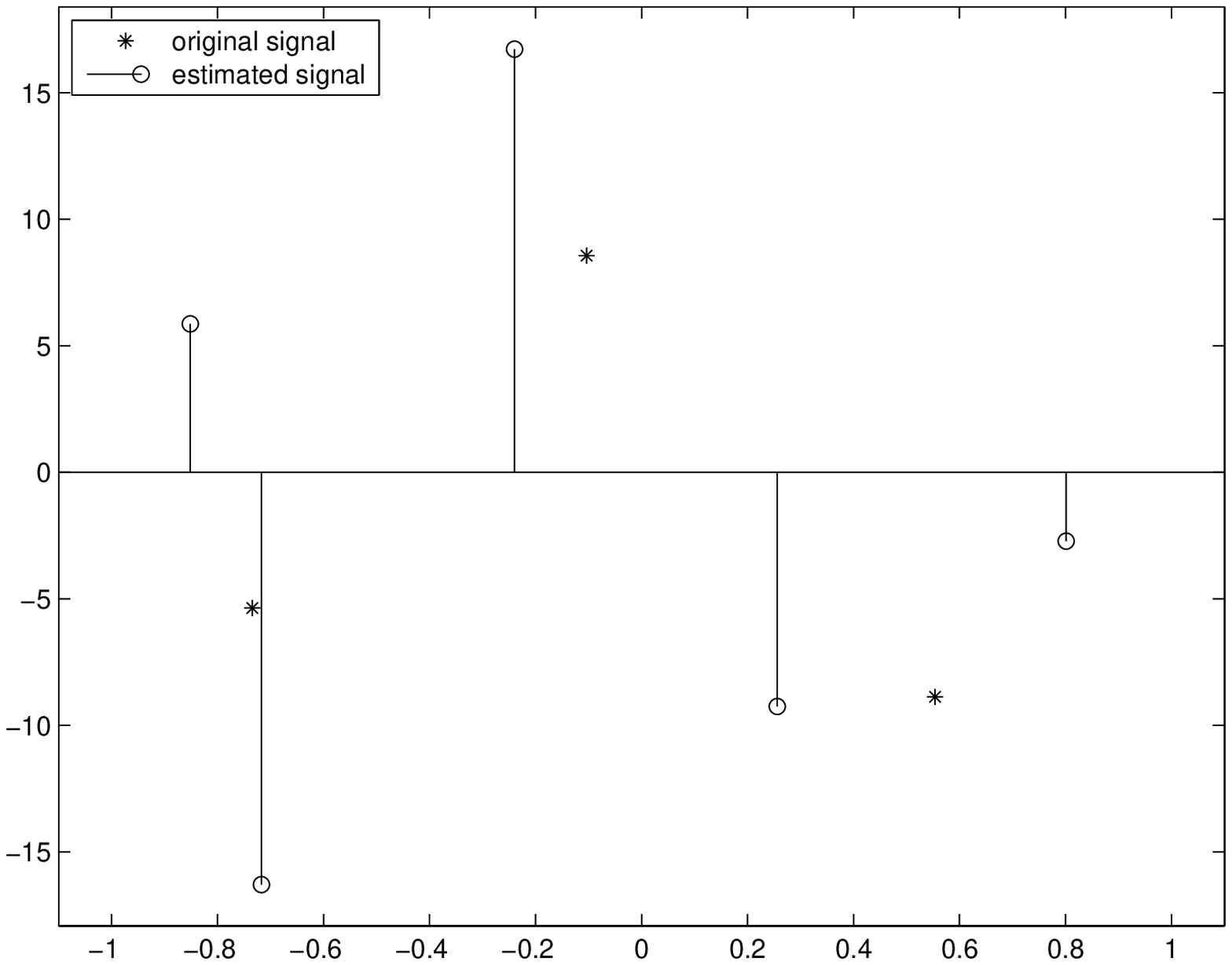}
\end{center}
\caption{Top : estimated spline (thick black line) of a non-uniform spline $\mbf f$ (thick dashed gray line) and its knots from a polynomial approximation (thin black line). Bottom : corresponding $d+1$-derivative of the spline (stars: original spline; circles: reconstructed spline). Degree $d = 3$, number of observed noisy moments $m - d = 7$. Noise levels $\sigma = \sigma_0 \frac{m!}{(m-d-1)!}$ with $\sigma_0 \equiv 0.0005, 0.002, 0.01$ (from left to right).}
\label{fig:Example4}
\end{figure}

%%%%%%%%%%%%%%%%%%%%%%%%%%%%%%%%%%%%%%%%%%%%%%%%%%%%%%%%%%%%%%%%%%%%%%
%\section{Important examples}
%
\appendix
\renewcommand*{\thesection}{\Alph{section}}
\section{Rice method}
\label{app:Rice}
\ni
Define the Gaussian process $\{X_{ m,{ d}}(t),\ t \in[-1,1]\}$ by:
\[
  \forall t\in[-1,1],\quad X_{ m,{ d}}(t) =  \xi_{ d+1}\varphi_{ d+1}(t) + \xi_{ d+2} \varphi_{ d+2}(t) +   \ldots +  \xi_{ m} \varphi_{{ m}}(t)\,,
\]
 where $ \xi_{ d+1},  \ldots ,  \xi_{ m}$ are i.i.d. standard normal. Its covariance function is:
 \[
 r(s,t) = \varphi_{ d+1}(t)\varphi_{ d+1}(s) + \varphi_{ d+2}(t)\varphi_{ d+2}(s) + \ldots +  \varphi_{ m}(t)\varphi_{{ m}}(s)\,,
 \]
where the dependence in ${ m}$ and ${ d}$ has been omitted. Observe its maximal variance is attained at point $1$ and is given by $\sigma^2_{ m,{ d}} = 2({ m}-{ d})$, and its variance function is $\sigma^2_{ m,{ d}} (t) =  \varphi_{{ d+1}}(t)^{2} + \varphi_{{ d+2}}(t)^{2} + \ldots +  \varphi_{{ m}}(t)^{2}$.
\begin{lem}
\label{lem:Rice0}
%Assume $ m\geq  d+1$. 
Let $\displaystyle\mathscr X = \max_{t \in [-1,1]} | X_{ m, d}(t)|$, then:
\[
  \forall u >\sqrt{2( m- d)},\quad\mathds P\{ \mathscr X>u\}  \leq  5( m+ d+1)\exp\Big[-\frac{u^{2}}{8( m- d)}\Big]\,.
  \]
\end{lem}
\begin{proof}
By the change of variables $t=\cos\theta$, for all $t\in[-1,1]$:
\[
X_{ m,{ d}}(t) =X_{ m,{ d}}(\cos\theta)= \sqrt2\xi_{ d+1} \cos(({ d}+1)\theta) + \ldots +  \sqrt2\xi_{ m} \cos( m\theta).
\]
Set $T(\theta) :=X_{ m, d}(t)$. We recall that its variance function is given by:
\[
\sigma^{2}_{ m, d}(\theta)=2\,\cos^{2}(( d+1)\theta) +  \ldots +  2\,\cos^{2}( m\theta)= m- d+\frac{\mathbf D_{ m}(2\theta)-\mathbf D_{ d}(2\theta)}2\,,
\]
where $\mathbf D_{k}$ denotes the Dirichlet kernel of order $k$. Observe that:
 \[
 \forall \theta\in\R\,,\quad\sigma^{2}_{ m, d}(\theta)\leq \sigma^{2}_{ m, d}(0)=2( m- d)\,,
 \]
%so that $T_m\big(\frac{3\pi}{2(1+2m)}\big)\sim\mathcal N(0,\frac{1+2m}3)$. 
By the Rice method \cite{Aza20081190}, for $u>0$:
 \begin{align*}
 \P\{\mathscr X>u\}& \leq 2  \P\{   \max_{\theta \in [0,\pi]} T(\theta)  >u\} \,,\\
& \leq
 2  \P\{T(0)>u\} +2\, \E[ U_u([0,\pi])]\,,\\
 &=   2\left[1-\Psi\Big(\frac{u}{\sqrt{2( m- d)}}\Big)\right]+  2 \int_{0}^\pi \E\big( (T'(\theta))^+ \big| T(\theta) =u) \psi_{\sigma_{ m, d}(\theta)} (u)\d \theta%\,,
 \end{align*}
 where $U_u$ is the number of crossings of the level $u$, $\Psi $ is the c.d.f. of the standard normal distribution, and $\psi_{\sigma}$ is the density of the centered normal distribution with standard error $\sigma$. First, observe that for $v>0$, $(1-{\Psi}(v) )\leq \frac{1}{2}\exp(-v^2/2) $. Hence, 
 \[
 1-\Psi\Big(\frac{u}{\sqrt{2( m- d)}}\Big)\leq \frac{1}{2} \exp\big(-\frac{u^{2}}{4( m- d)}\big)\,.
 \]
 Moreover, regression formulas implies that:
 \begin{align*}
  \E\big( T'(\theta)\big| T(\theta) =u\big) & =\frac{ r_{0,1} (\theta,\theta)}{r(\theta,\theta)} u\,, \\
  \Var\big( T'(\theta)\big| T(\theta) =u\big) & \leq   \Var\big( T'(\theta)\big) =  r_{1,1} (\theta,\theta)\,,
  \end{align*}
  where, for instance,  $r_{1,1} (\nu,\theta) =\frac{ \partial^2 r(\nu,\theta)}{\partial \nu \partial \theta}$. We recall that the covariance function is given by:
  \begin{align*}
   r(\nu,\theta)&=2\cos(( d+1)\nu)\cos(( d+1)\theta) +  \ldots +  2\cos( m\nu)\cos( m\theta)\,, \\
   & = \frac 12 \big[\mathbf D_{ m}(\nu-\theta)+\mathbf D_{ m}(\nu+\theta)-\mathbf D_{ d}(\nu-\theta)-\mathbf D_{ d}(\nu+\theta) \big]\,.
  \end{align*}
Observe that:
 \begin{align*}
  r_{0,1} (\theta,\theta) & = \frac12[\mathbf D'_{ m}(2\theta)-\mathbf D'_{ d}(2\theta)]=-\sum_{k= d+1}^{ m}k\sin(2k\theta)\,,\\
    r_{1,1} (\theta,\theta) & =\sum_{k= d+1}^{ m}k^{2}(1-\cos(2k\theta))\,.
   \end{align*}
 On the other hand, if $Z \sim \mathcal N(\mu, \sigma^2)$ then
\[
 \E(Z^+ ) = \mu\, \Psi\big( \frac\mu\sigma\big) + \sigma\,  \psi\big( \frac\mu\sigma\big) \leq  \mu^+ + \dfrac\sigma{\sqrt{2 \pi}}\,,
\]
where $\psi$ is the standard normal density. We get  that:
 \begin{align*}
 &\int_{0}^\pi \E\big( (T'(\theta))^+ \big| T(\theta) =u) \psi_{\sigma_{ m, d}(\theta)} (u)\d t\\
 &\leq\int_{0}^\pi\frac{ [\mathbf D'_{ m}(2\theta)-\mathbf D'_{ d}(2\theta)]^{+}}{2\,\sigma^{2}_{ m, d}(\theta)}\, u \psi_{\sigma_{ m, d}(\theta)}(u) \d\theta\\&  + \frac{1}{\sqrt{2\pi}} \int_{0}^\pi\big[\sum_{k= d+1}^{ m}k^{2}(1-\cos(2k\theta))\big]^{1/2} \,\psi_{\sigma_{ m, d}(\theta)}(u) \d\theta\,, \\
  &= A+B\,.
\end{align*}
\ni
We use the following straightforward relations: 
\begin{itemize}
  \item $\forall\,0<\sigma_1 <\sigma_2 <u\,,\quad\psi_{\sigma_1}(u) \leq \psi_{\sigma_2}(u)$, 
  \item $\forall\theta\,,\quad[\mathbf D'_{ m}(2\theta)-\mathbf D'_{ d}(2\theta)]^{+}\leq \sum_{k= d+1}^{ m}k=\frac{( m+ d+1)( m- d)}2$,
    \item $\forall\theta \in [0,\pi]$,
\[
  \frac{ u}{2\,\sigma^{2}_{ m, d}(\theta)}\,  \psi_{\sigma_{ m, d}(\theta)}(u)\leq  \frac{1}{2\sqrt{2\pi}u^2} \,\frac{ u^3}{\,\sigma^3_{ m, d}(\theta)}  e^{- \frac{ u^2}{4\,\sigma^2_{ m, d}(\theta)} } \,e^{- \frac{ u^2}{4\,\sigma^2_{ m, d}(\theta)} }\leq  \frac{2}{3u^2}e^{-\frac{u^2}{8 ( m- d)}}\,.
\]
\end{itemize}
\ni
Eventually, we get, for $u> \sqrt{2( m- d)}$:
\begin{align*}
A &\leq  \frac{\pi}3\, \frac{( m+ d+1)( m- d)}{u^{2}}\exp(-\frac{u^{2}}{8( m- d)})\,,\\
B&  \leq \Big[\frac\pi{12}((2 m+1)( m+1) m-(2 d+1)( d+1) d)\Big]^{1/2}\,\psi_{\sqrt{2( m- d)}} (u)\,.
\end {align*}
and the result follows.
\end{proof}
\ni
As a corollary, we deduce Lemma \ref{lem:Rice}.

\section{Fenchel dual and first order conditions}
\label{app:Fenchel}
\begin{lem}\label{lem:ConvexConjugate}
The program:
\eq
\label{Prog:Primal}
 \inf_{ \mu\in\mbf C_{ d}(\mbf x)}\frac12\lVert \mbf c(\mu)-\mbf y\lVert^2_2+\lambda\lVert\mu\lVert_{TV}\,,
\qe
has Fenchel dual program:
\eq\label{Prog:Dual}
-\inf_{\displaystyle\lVert\sum_{k=0}^{ m}\alpha_{k}\varphi_{k}\lVert_{\infty}\leq\lambda}\Big\{\langle\alpha,y\rangle+\frac12\sum_{k= d+1}^{ m}\alpha_{k}^{2}\Big\}\,.
\qe
Moreover, there is no duality gap.
\end{lem}
\begin{proof} The case $ d=-1$ has been treated in \cite{ azais2014spike}. Assume that $ d\geq0$. Program \eqref{Prog:Primal} can be viewed as:
\eq\notag%\label{Prog:PrimalSomme}
\inf_{\mu\in \mathcal M}\quad h(\mbf c(\mu))+\psi_{1}(\mu)+\psi_{2}(\mu)\,,
\qe
where $h(c):=(1/2)\lVert c-\mbf y\lVert_{2}^{2}$, $\psi_{1}(\mu):=\lambda\lVert\mu\lVert_{TV}$ and $\psi_{2}(\mu):=\imath\!\imath_{\mbf C_{ d}(\mbf x)}(\mu)$, with:
\[
\imath\!\imath_{\mbf C_{ d}(\mbf x)}(\mu)=
\Bigg\{
\begin{array}{ll}
0&\mathrm{if\ }\mu\in\mbf C_{ d}(\mbf x)\,, \\
\infty&\mathrm{otherwise}\,. 
\end{array}
\]
Note the function $h$ has Legendre conjugate:
\[
\forall\alpha\in\mathds R^{ m+1},\quad h^{\star}(\alpha)=\langle\alpha,\mbf y\rangle+\frac12\lVert\alpha\lVert_{2}^{2}\,,
\]
One can check that the function $\psi_{1}$ has Legendre conjugate:
\[
\forall f\in\mathcal C([-1,1]),\quad\psi_{1}^{\star}(f)=\sup_{\mu\in \mathcal M}\int f\d\mu-\lambda\lVert\mu\lVert_{TV}=\imath\!\imath_{\mathrm B_{\infty}(\lambda)}(f)\,,
\]
where:
\[
\imath\!\imath_{\mathrm B_{\infty}(\lambda)}(f)=
\Bigg\{
\begin{array}{ll}
0&\mathrm{if\ }\lVert f\lVert_{\infty}\leq\lambda\,, \\
\infty&\mathrm{otherwise}\,. 
\end{array}
\]
Indeed, we have, for all $\mu\in \mathcal M$, $\int f\d\mu-\lambda\lVert\mu\lVert_{TV} \leq (\lVert f\lVert_{\infty} - \lambda ) \lVert\mu\lVert_{TV}$, showing that the supremum over $\mu$ is $0$ if $\lVert f\lVert_{\infty} \leq \lambda$. If $\lVert f\lVert_{\infty} > \lambda$, define, for every $A>0$, $\mu_A = A \; \text {sg}( f(x_0) ) \;\delta_{x_0}$ where $x_0$ is such that $\lVert f\lVert_{\infty} = |f(x_0)|$. Then $\int f\d\mu_A-\lambda\lVert\mu_A \lVert_{TV}= A (\lVert f\lVert_{\infty} - \lambda)$ for every $A>0$, which completes proving our claim.

Let us turn to the Legendre conjugate of $\psi_2$. We show that
\[
\forall f\in\mathcal C([-1,1]),\quad
\psi_{2}^{\star}(f)=\sup_{\mu\in\mbf C_{ d}(\mbf x)}\int f\d\mu=
\left\{
\begin{array}{ll}
\displaystyle\sum_{k=0}^{d}a_{k}y_{k}&\mathrm{if\ }f=\displaystyle\sum_{k=0}^{d}a_{k}\varphi_{k}\,, \\
\infty&\mathrm{otherwise}\,. 
\end{array}
\right.
\]
Indeed, the result is obvious if $f$ is of the form $f=\displaystyle\sum_{k=0}^{d}a_{k}\varphi_{k}$. In the other case, recall that $\{\varphi_k\}_{k \geq 0}$ is a complete orthonormal family of $L^2([-1,1],\measure)$ where $\measure(\d t)=(1/\pi)\,({1-t^{2}})^{-1/2}\,\d t$ ($\d t$ denotes the Lebesgue measure). Thus, in this Hilbert space, $f$ can be expanded as $f = \sum_{k=0}^\infty a_k \varphi_k$ with $a_p \neq 0$ for som $p >d$. Define the measure $\mu_1 (dt) = \varphi_p(t) \measure (dt)$. Observe that $\int \varphi_k d\mu_1 = \delta_{kp}$ and $\int f d\mu_1 = a_p$. Let $\mu_0 \in \mbf C_{ d}(\mbf x)$, and $\mu_A = \mu_0 + A \, \mu_1$ for every $A \in \mathbb R$. Then $\mu_A \in \mbf C_{ d}(\mbf x)$ and $\int f d\mu_A = \sum_{k=0}^d a_k y_k + A\, a_p, \, \forall A \in \mathbb R$. This proves our claim.

%Indeed, if $f=\displaystyle\sum_{k=0}^{d}a_{k}\varphi_{k}$ then $\int f\d\mu =\displaystyle\sum_{k=0}^{d}a_{k}y_{k}$ (recall $y_{k}=c_{k}(\mbf x)$ for $0\leq k\leq d$). If $f$ does not belong to the span of $\{\varphi_{1},\ldots,\varphi_{d}\}$ then 
Let $f\in\mathscr C([-1,1])$. The Legendre conjugate $\psi^{\star}$ of $\psi:=\psi_{1}+\psi_{2}$ at $f$ is given by:
\eq\label{eq:PsiStarValue}
\psi^{\star}(f)=\inf_{f=f_{1}+f_{2}}\psi_{1}^{\star}(f_{1})+\psi_{2}^{\star}(f_{2})=\inf_{a_{0},\ldots,a_{d}\in\R}\imath\!\imath_{\mathrm B_{\infty}(\lambda)}(f-\sum_{k=0}^{d}a_{k}\varphi_{k})+\sum_{k=0}^{d}a_{k}y_{k}\,.
\qe
Indeed, observe that the bi-conjugate of $\psi_{1}$ (resp. $\psi_{2}$) enjoys $\psi_{1}^{\star\star}=\psi_{1}$ (resp. $\psi_{2}^{\star\star}=\psi_{2}$) and it holds:
\begin{align*}
\psi^{\star}(f)&=\sup_{\mu\in\mathcal M}\{\int f\d\mu-\psi_{1}(\mu)-\psi_{2}(\mu)\}\,,\\
&=\sup_{\mu\in\mathcal M}\{\int f\d\mu-\sup_{f_{1}}\{\int f_{1}\d\mu-\psi^{\star}_{1}(f_{1})\}-\sup_{f_{2}}\{\int f_{2}\d\mu-\psi^{\star}_{2}(f_{2})\}\}\,,\\
&=\inf_{f_{1},f_{2}}\sup_{\mu\in\mathcal M}\{\int (f-f_{1}-f_{2})\d\mu+\psi^{\star}_{1}(f_{1})+\psi^{\star}_{2}(f_{2})\}\,,\\
&=\inf_{f=f_{1}+f_{2}}\psi^{\star}_{1}(f_{1})+\psi^{\star}_{2}(f_{2})\,.
\end{align*}

\noindent
Moreover, observe that the dual operator $\mbf c^{\star}$ of $\mbf c$ is given by:
\[
\forall\alpha\in\mathds R^{ m+1}\,,\quad\mbf c^{\star}(\alpha)=\sum_{k=0}^{ m}\alpha_{k} \varphi_{k}\,.
\]
Observe that the bi-conjugate of $h$ enjoys $h^{\star\star}=h$. Then, notice that:
\begin{align*}
 \inf_{\mu\in \mathcal M}\ h(\mbf c(\mu))+\psi(\mu)
 &=\inf_{\mu\in \mathcal M}\sup_{\alpha\in\R^{m+1}}\{\langle\alpha,\mbf c(\mu)\rangle-h^\star(\alpha)\}+\psi(\mu)\,,\\
 &=\sup_{\alpha\in\R^{m+1}}-h^\star(\alpha)-\sup_{\mu\in \mathcal M}\{\langle-\mbf c^\star(\alpha),\mu\rangle-\psi(\mu)\}\,,\\
 &=-\inf_{\alpha\in\mathds R^{ m+1}}h^{\star}(\alpha)+\psi^{\star}(-\mbf c^{\star}(\alpha))\,.
\end{align*}
It follows that the program \eqref{Prog:Primal} has Fenchel dual:
\[
-\inf_{\alpha\in\mathds R^{ m+1}}h^{\star}(\alpha)+\psi^{\star}(-\mbf c^{\star}(\alpha))=-\inf_{\lVert\sum_{k=0}^{m}\alpha_{k}\varphi_{k}\lVert_{\infty}\leq\lambda}\Big\{\langle\alpha,\mbf y\rangle+\frac12\sum_{k= d+1}^{ m}\alpha_{k}^{2}\Big\}\,.
\]
Slater's condition shows that strong duality holds.
\end{proof}

\begin{lem} The first order conditions read: There exists $\hat a_{0},\ldots,\hat a_{d}\in\R$ such that
\eq
\label{eq:OptimalityPremierOrdre}
\lVert\hat P\lVert_{\infty}\leq\lambda\quad \mathrm{and}\quad \lambda\lVert\hat{\mbf x}\lVert_{TV}=\int_{-1}^{1}\hat P\mathrm d(\hat{\mbf x})\,,
\qe
where:
\[
\hat P=\sum_{k=0}^{d}\hat a_{k}\varphi_{k}+\sum_{k=d+1}^{ m}(y_{k}-c_{k}(\hat{\mbf x}))\varphi_{k}\,.
\]
\end{lem}

\begin{proof}
Let $\mu\in\mbf C_{ d}(\mbf x)$ and $\gamma\in(0,1)$. Set $\nu=\hat{\mbf x}+\gamma(\mu-\hat{\mbf x})$ then, by convexity:
\[
\lVert\mu\lVert_{TV}-\lVert\hat{\mbf x}\lVert_{TV}\geq\frac1\gamma(\lVert\nu\lVert_{TV}-\lVert\hat{\mbf x}\lVert_{TV})\,.
\]
Observe that $\nu\in\mbf C_{ d}(\mbf x)$, by optimality:
\begin{align*}
\lambda(\lVert\nu\lVert_{TV}-\lVert\hat{\mbf x}\lVert_{TV})&\geq\frac12(\lVert \mbf c(\hat{\mbf x})-\mbf y\lVert^2_2-\lVert \mbf c(\nu)-\mbf y\lVert^2_2)\,,\\
&=\gamma\langle \mbf y-\mbf c(\hat{\mbf x}),\mbf c(\mu)-\mbf c(\hat{\mbf x})\rangle-\frac{\gamma^{2}}2\lVert \mbf c(\mu)-\mbf c(\hat{\mbf x})\lVert^2_2\,.
\end{align*}
Letting $\gamma$ go to $0$, we deduce:
\eq
\label{eq:OptimalityProgram1}
\forall\mu\in\mbf C_{ d}(\mbf x)\,,\quad\lambda(\lVert\mu\lVert_{TV}-\lVert\hat{\mbf x}\lVert_{TV})\geq\langle \mbf y-\mbf c(\hat{\mbf x}),\mbf c(\mu)-\mbf c(\hat{\mbf x})\rangle\,.
\qe
Conversely, if \eqref{eq:OptimalityProgram1} holds then, for all $\mu\in\mbf C_{ d}(\mbf x)$:
\begin{align*}
\frac12\lVert \mbf c(\mu)-\mbf y\lVert^2_2+\lambda\lVert\mu\lVert_{TV}\geq&\frac12\lVert \mbf c(\hat{\mbf x})-\mbf y+\mbf c(\mu)-\mbf c(\hat{\mbf x})\lVert^2_2\\&+\langle \mbf y-\mbf c(\hat{\mbf x}),\mbf c(\mu)-\mbf c(\hat{\mbf x})\rangle+\lambda\lVert\hat{\mbf x}\Vert_{TV}\,,\\
=&\frac12\lVert \mbf c(\hat{\mbf x})-\mbf y\lVert^2_2+\lambda\lVert\hat{\mbf x}\lVert_{TV}+\frac12\lVert \mbf c(\mu)-\mbf c(\hat{\mbf x})\lVert^2_2\,,\\
\geq&\frac12\lVert \mbf c(\hat{\mbf x})-\mbf y\lVert^2_2+\lambda\lVert\hat{\mbf x}\lVert_{TV}\,.
\end{align*}
Therefore, Eq. \eqref{eq:OptimalityProgram1} is a necessary and sufficient condition for the measure $\hat{\mbf x}$ to be a solution to \eqref{def:Blasso}. In particular, it follows:
\[
\lambda\lVert\hat{\mbf x}\lVert_{TV}-\langle \mbf y-\mbf c(\hat{\mbf x}),\mbf c(\hat{\mbf x})\rangle\leq\inf_{\mu\in\mbf C_{ d}(\mbf x)}\{\lambda\lVert\mu\lVert_{TV}-\langle \mbf y-\mbf c(\hat{\mbf x}),\mbf c(\mu)\rangle\}=-\psi^{\star}(\hat Q)\,,
\]
where $\psi^{\star}$ is defined by \eqref{eq:PsiStarValue} and $\displaystyle\hat Q=\sum_{k= d+1}^{ m}(y_{k}-c_{k}(\hat{\mbf x}))\varphi_{k}$. The optimality conditions can be deduced from \eqref{eq:PsiStarValue}.
\end{proof}

\section{Proof of Lemma \ref{lem:MajorationPrediction}}
\label{app:MajorationPrediction}
\ni
Let $(a_k)_{k=0}^{ m}$ be the coefficients of $P$, namely:
\[
P=\sum_{k=0}^{ m}a_{k}\varphi_{k}\,.
\]
Since $\mathscr F$ is an orthonormal family of $L^{2}(\measure)$, it holds
\begin{align*}
\int_{-1}^1P\mathrm{d}(\hat{\mbf x}-\mbf x)&=\sum_{k=0}^{ m}a_{k}\int_{-1}^1\varphi_{k}\mathrm{d}(\hat{\mbf x}-\mbf x)\,,\\
&=\sum_{k=d+1}^{ m}a_{k}(c_{k}(\hat{\mbf x})-c_{k}(\mbf x))\,,\\
&=\int_{-1}^1(\sum_{k=d+1}^{ m}a_{k}\varphi_{k})(\sum_{k=d+1}^{ m}(c_{k}(\hat{\mbf x})-c_{k}(\mbf x))\varphi_{k})\mathrm d \measure\,,\\
&=\int_{-1}^1(\sum_{k=d+1}^{ m}a_{k}\varphi_{k})(-\sum_{k=0}^{d}\hat a_{k}\varphi_{k}+\sum_{k=d+1}^{ m}(c_{k}(\hat{\mbf x})-c_{k}(\mbf x))\varphi_{k})\mathrm d \measure\,,\\
&=\int_{-1}^1(\sum_{k=d+1}^{ m}a_{k}\varphi_{k})(\sum_{k=0}^{ m}\varepsilon_{k}\varphi_{k}-\hat P)\mathrm d \measure\,,
\end{align*}
where $\hat a_{0},\ldots,\hat a_{d}\in\R$ and $\hat P$ are given by Lemma \ref{eq:OptimalityPremierOrdre}. By Cauchy-Schwarz inequality, it yields:
\begin{align*}
\int_{-1}^1P\mathrm{d}(\hat{\mbf x}-\mbf x)&\leq\lVert\sum_{k=d+1}^{ m}a_{k}\varphi_{k}\lVert_{2}\, \lVert \sum_{k=0}^{ m}\varepsilon_{k}\varphi_{k}-\hat P\lVert_{2}\,,\\
&\leq \lVert P\lVert_{2}\, \lVert \sum_{k=0}^{ m}\varepsilon_{k}\varphi_{k}-\hat P\lVert_{\infty}\,,\\
&\leq \lVert P\lVert_{\infty}\, (\lVert \sum_{k=0}^{ m}\varepsilon_{k}\varphi_{k}\lVert_{\infty}+\lVert\hat P\lVert_{\infty})\,,
\end{align*}
where $\lVert \,.\,\lVert_{2}$ stands for the norm associated to the Hilbert space $L^{2}(\measure)$ for which $\mathcal F$ is an orthonormal family. The result follows from \eqref{eq:OptimalityPremierOrdre}.

\section{Background on Semi-Definite Programming in Super-Resolution}
\label{app:Back}
\subsection*{Zero-noise problem}
In the noiseless case, observe that $\mbf n=0$. Exact recovery from moment samples has been investigated in \cite{azais2014spike,bendory2013exact} where one considers the program:
\begin{equation}\label{prog:supportpursuit}%\tag{$\mathrm{GME}$}
 \mbf x^{0}\in\arg\min\displaylimits_{\mu\in m}\ \lVert{\mu}\lVert_{TV}\quad \text{s.t.} \
\int \Phi\,\d\mu=\int \Phi\,\d\mbf x\,,
\end{equation}
where $\Phi=(\varphi_{0},\ldots,\varphi_{ m})$ is the Chebyshev moment curve. The optimality condition of \eqref{prog:supportpursuit} shows that the sub-gradient of the $\ell_{1}$-norm vanishes at any solution point $\mbf x^{0}$. Therefore a sufficient condition for exact recovery is that $\mbf x$ satisfies the optimality condition. This is covered by the notion of ``dual certificate'' \cite{de2012exact,candes2012towards} or equivalently the notion of ``source condition'' \cite{burger2004convergence}.

\begin{defn}[Dual certificate]
We say that a polynomial $P=\sum_{k=0}^{m}\alpha_k\varphi_{k}$ is a dual certificate for the measure $\mbf x$ defined by \eqref{def:TargetMeasure} if and only if it satisfies the following properties:
 \begin{itemize}
  \item sign interpolation:  $\forall k\in\{1,\ldots,S\}\,,\ P(\mbf t_k)=a_{k}/\lvert a_{k}\lvert$,
  \item $\ell_{\infty}$-constraint: $\lVert P\lVert_{\infty}\leq1$.
 \end{itemize}
\end{defn}
\ni
One can prove \cite{de2012exact} that $\mbf x$ is a solution to \eqref{prog:supportpursuit} {if and only if} $\mbf x$ has a dual certificate. 
\subsection*{Semi-noisy moment sample model} In our model, we deal with an observation $\mbf y$ described by \eqref{eq:defObservation}. In this case, the existence of a dual certificate is not sufficient to derive support localization, see for instance \cite{ azais2014spike}. One needs to strengthen this notion using the Quadratic Isolation Condition \cite{ azais2014spike}.

\begin{defn}[Quadratic isolation condition] 
 \label{def:QIC}
A finite set $\mathbf T=\{ t_1,\ldots, t_s\}\subset[-1,1]$ satisfies the quadratic isolation condition with parameters $C_{a}>0$ and $0<C_{b}<1$, denoted by $\mathrm{QIC}(C_{a}, C_{b})$, if and only if for all $\{\theta_k\}_{k=1}^{s}\in\mathds R^{s}$, there exists $P\in\mathrm{Span}(\mathscr F)$ such that  for all $k=1,\ldots,s$, $P(\mbf t_k)=\exp(-\mathbf{i}\theta_k)$, and
\[\forall x\in[-1,1]\,,\quad 1-\lvert P(x)\lvert\geq \min_{t\in\mathbf T}\ \{C_{a} m^{2}d(x,t)^{2},C_{b}\}\,.\]
\end{defn}
\ni
As showed by Lemma \ref{lem:QIC}, if the support $\mathbf T$ satisfy a minimal separation condition described in \eqref{hyp:separation} then $\mathbf T$ satisfies $\mathrm{QIC}(C_{a}, C_{b})$ with constants $C_{a}= 0.00848$ and $C_{b}= 0.00879$.

%\section{ main result}
\subsection*{Semi-definite programming}
%%%%%%%%%%%%%%%%%%%%%%%%%%%%%%%%%%%%%%%%%%%%%%%%%%%%%%%%%%%%%%%%%%%%%%
\ni
Observe that the Fenchel dual program of \eqref{def:Blasso} is given by:
\eq\label{def:DualBlasso}
\hat\alpha\in\arg\min_{\displaystyle\lVert\sum_{k=0}^{ m}\alpha_{k}\varphi_{k}\lVert_{\infty}\leq\lambda}\Big\{\langle\alpha,y\rangle+\frac12\sum_{k= d+1}^{ m}\alpha_{k}^{2}\Big\}\,,
\qe
and strong duality holds, see Lemma \ref{lem:ConvexConjugate}. %This dual program can be seen as the orthogonal projection of the vector $(0,\ldots,0,y_{ d+1},\ldots,y_{ m})$ onto the convex set:
%\[
%\big\{\alpha\in\mathds R^{ m+1}\,,\quad\lVert\sum_{k=0}^{ m}\alpha_{k}\varphi_{k}\lVert_{\infty}\leq\lambda\big\}\,.
%\]
%Therefore, there is a unique solution to \eqref{def:DualBlasso}. 
Moreover, observe that the constraint  $\lVert\sum_{k=0}^{ m}\alpha_{k}\varphi_{k}\lVert_{\infty}\leq\lambda$ can be re-cast as imposing that the algebraic polynomials:
\eq\label{eq:LInfConstraints}
P_{1}:=\lambda+\sum_{k=0}^{ m}\alpha_{k}\varphi_{k}\geq0\quad \mathrm{and}\quad P_{2}:=\lambda-\sum_{k=0}^{ m}\alpha_{k}\varphi_{k}\geq0\,.
\qe
Considering the change of variables $\theta=\arccos(t)$, the aforementioned inequalities can be equivalently drawn for some trigonometric polynomials. Using Riesz-Fej\'er theorem, one can show that non-negative trigonometric polynomials are sums of squares polynomials (SOS). A standard result, see for instance \cite{dumitrescu2007positive}, ensures that the convex set of sum of square polynomials (SOS) can be described as the intersection between the set of positive hermitian semi-definite (SDP) matrices and an affine constraint.

\begin{lem}%[see \cite{dumitrescu2007positive} for instance]
The constraint \eqref{eq:LInfConstraints} can be re-casted into a semi-definite constraint.
\end{lem}
\ni
Hence, we can compute $\hat\alpha$ using a SDP program. Moreover, Fenchel's duality theorem shows that the dual polynomial:
\[
\hat P=\frac1\lambda\sum_{k=0}^{ m}\hat\alpha_{k}\varphi_{k}\,,
\]
is a sub-gradient of the $TV$-norm at point $\hat{\mbf x}$. In particular, the support $\hat{\mbf T}$ of $\hat{\mbf x}$ is included in:
\[
\big\{t\in[-1,1]\,,\quad\lvert \hat P\lvert=1\big\}\,.
\]
If $\hat P$ is not constant, this level set has at most $ m+1$ points and it defines the support of the solution. Hence, we can find the weights of $\hat{\mbf x}$ using a least-square-type estimator subject to the affine constraint given by the intersection between $\mbf C_{ d}(\mbf x)$ and discrete measures with support included in $\hat{\mbf T}$. In this case, the solution to \eqref{def:Blasso} is unique and can be computed using the aforementioned SDP program. If $\hat P$ is constant then there always exists a solution to \eqref{def:Blasso} with finite support. Indeed, using the fact that there is no duality gap, one can check that the solution has non-negative (resp. non-positive) weights if $\hat P=1$ (resp. $\hat P=-1$). Therefore, Carath\'eodory's theorem shows that there always exists a solution with finite support\footnote{The interested reader may find a valuable reference on the geometry of the cone of non-negative measures in \cite{krein1977markov}.}. However, one can not use the dual program \eqref{def:DualBlasso} to compute the solution to the primal program \eqref{def:Blasso}. We deduce the following lemma.
\begin{lem}
\label{lem:SolDiscrete}
There always exists a solution to the primal problem \eqref{def:Blasso} with a support of size at most $ m+2$. Moreover, if $\hat P$ is not constant, the solution to \eqref{def:Blasso} is unique, its support is included in the level set $\{t\in[-1,1]\,,\ \lvert \hat P\lvert=1\}$ and has size at most $ m+1$.
\end{lem}

%%%%%%%%%%%%%%%%%%%%%%%%%%%%%%%%%%%%%%%%%%%%%%%%%%%%%%%%%%%%%%%%%%%%%%

\bibliographystyle{elsarticle-harv}
 \bibliography{biblio}
\end{document}